\documentclass[12pt]{amsart}
\usepackage{amssymb,amsmath,amsfonts,latexsym}
\usepackage{hyperref}
\newtheorem{theorem}{Theorem}[section]
\newtheorem{lemma}[theorem]{Lemma}
\newtheorem{proposition}[theorem]{Proposition}
\newtheorem{corollary}[theorem]{Corollary}
\theoremstyle{definition}

\newtheorem{example}[theorem]{Example}

\newtheorem{remark}[theorem]{Remark}
\numberwithin{equation}{section}

\begin{document}

\title[On Rational Orbits in Some Prehomogeneous Vector Spaces]{On Rational Orbits in Some Prehomogeneous Vector Spaces}

\author[Sayan Pal]{Sayan Pal}

\address{Indian Statistical Institute, Statistics and Mathematics Unit, 8th Mile, Mysore Road, Bangalore, 560059, India}

\email{syn.pal98@gmail.com}

\begin{abstract} \noindent Let $k$ be a field with characteristic different from $2$. In this paper, we describe the $k$-rational orbit spaces in some irreducible prehomogeneous vector spaces $(G,V)$, where $G$ is a connected reductive algebraic group and $V$ is an irreducible rational representation of $G$ with a Zariski dense open orbit over the field $k$. We prove that all composition algebras over $k$ appear as orbit spaces in some of these representations associated to the group $Sp_{6}$. This leads to a parametric description of the reduced Freudenthal algebras of dimensions $6$ and $9$ over $k$ (if $\text{char}(k)\neq 2,3$). In this process, we provide a construction of the irreducible representation of $Sp_{6}$ with dimension $14$ from the split octonion algebra. We also get the orbit decomposition in the irreducible representation of dimension $7$ and the corresponding octonion algebra of any simple connected group of type $G_{2}$, and a parametrization of the isotopes of any Freudenthal algebra over $k$.


\end{abstract}

\keywords{Prehomogeneous vector spaces, symplectic groups, composition algebras, Freudenthal algebras}

\subjclass{17A75, 17C60, 20G05, 20G15}

\maketitle

\tableofcontents

\section{Introduction}\label{I}

The main aim of this paper is to explore the $k$-rational orbit spaces in some prehomogeneous vector spaces over a field $k$. We prove that each orbit from the orbit spaces in those prehomogeneous vector spaces uniquely represents several algebraic structures like the composition algebras and the reduced Freudenthal algebras of dimensions $6$ and $9$. Conversely, we can get all of these algebraic structures (up to isomorphism) from these orbit spaces. In general, we get a surjection from these orbit spaces onto the set of all isomorphism classes of the algebras mentioned above. But we will see later that in some cases these surjections are not bijective (see Section \ref{4}, \ref{5}). In addition to the new results, we also get alternate proofs and generalizations of some existing results, which we proceed to describe below briefly.

\vspace{2 mm}

Let $G$ be a connected reductive algebraic group and $V$ be a rational representation of $G$, defined over a field $k$. Assume that $(G,V)$ is a prehomogeneous vector space (henceforth called a PV) with some relative invariant $f\in k[V]$, i.e., $G$ has a Zariski dense open orbit in $V$. Let $V^{ss}$ be the complement of the hypersurface in $V$ defined by the relative invariant $f$, which is called the set of \textit{semi-stable points}. Then $(V^{ss})^{c}=\{ x\in V:f(x)=0\}\subset V$, the hypersurface defined by $f$. By $V_{k}$, $V^{ss}_{k}$ and $G_{k}$ we denote the set of all $k$-rational points in $V$, $V^{ss}$ and $G$, respectively. One expects the orbit space $V^{ss}_{k}/G_{k}$ to correspond to some arithmetic objects; for example, see (\cite{WY}) and (\cite{YA}). In these two papers, A. Yukie and David J. Wright have described those arithmetic objects for some prehomogeneous vector spaces. In many cases, the arithmetic interpretations of these orbit spaces are not known. Moreover, one may ask if the $G_{k}$-orbits in the hypersurface $(V^{ss}_{k})^{c}\subset V_{k}$ also parametrize similar kinds of arithmetic objects like the orbit space $V^{ss}_{k}/G_{k}$.

  \vspace{2 mm}

  Let $(C,N_{C})$ be any split composition algebra with quadratic norm $N_{C}$ defined over the field $k$, or the field $k$ itself (char$(k)\neq 2$). Then we have the reduced Freudenthal algebra $\mathcal{H}_{3}(C,I_{3})$, if $\text{char}(k)\neq 2,3$; and we can also define an algebra structure on the vector space \[Z(C)=\left(\begin{array}{cc}
k & \mathcal{H}_{3}(C,I_{3}) \\
\mathcal{H}_{3}(C,I_{3}) & k
\end{array}\right).\] Each of these $Z(C)$'s above is a structurable algebra with respect to a certain multiplication and a quartic form $J$ defined on it (see \cite{AB} for details). These algebras are irreducible representations of certain subgroups $G(C)\subset GL(Z(C))$, which leave the quartic form $J$ invariant (see \cite{BGL}, \cite{LM} for details). We describe all the cases in Table $\ref{1}$, where $V_{6}$ denotes the vector space of dimension $6$ and $\wedge^{3}_{0}V_{6}\subset \wedge^{3}V_{6}$ (see Section \ref{3} for details). We also refer the reader to have a look at (\cite{RG}).

 \begin{table}[h]
    \centering
    \caption{The structurable algebras as $G(C)$-representations}\label{1}
    \begin{tabular}{|c|c|c|c|} \hline 
         $\text{dim}(C)$ & $G(C)$&  $Z(C)$ & $\text{dim}(Z(C))$ \\ \hline 
         $1$& $Sp_{6}$ &  $\wedge_{0}^{3}V_{6}$&  $14$ \\ \hline 
         $2$& $SL_{6}$ &  $\wedge^{3}V_{6}$&  $20$\\ \hline 
         $4$& $Spin_{12}$&  half-spin representation&  $32$\\ \hline 
         $8$& $E_{7}$&  minuscule representation&  $56$\\ \hline
    \end{tabular}
\end{table}

\vspace{2 mm}

The cases where $\text{dim}(C)=4,8$, have been discussed in detail in (\cite{IJ}) and (\cite{HJ}), respectively. The second representation $\wedge^{3}V_{6}$ has also been discussed in \cite{PS} as a representation of $Sp_{6}(\subset SL_{6})$. In this paper, we discuss the case where $\text{dim}(C)=1$, i.e., $C=k$. In that case, $Z(C)\simeq \wedge_{0}^{3}V_{6}$ (see Section \ref{3} for the notation) is a $14$-dimensional faithful irreducible representation of $G(C)=Sp_{6}$ with $1$-dimensional ring of invariants $k[Z(C)]^{Sp_{6}}=k[J]$, generated by the $Sp_{6}$-invariant quartic form $J$ defined on $Z(C)$ (see \cite{BGL} for details). We will describe the representation and the invariant $J$ in detail later in Section \ref{3}. We denote $Z(C)$ for $C=k$ simply by $Z$, which is a structurable algebra over the field $k$ (provided $\text{char}(k)\neq 2,3$), as mentioned above. Since we are going to study the underlying vector space of the algebra $Z$ as a $Sp_{6}$-representation, we may assume that $\text{char}(k)\neq 2$ and consider $Z$ only as a vector space that is a $Sp_{6}$-representation with the action defined in Section \ref{3}. This $Sp_{6}$-representation $Z$ can be constructed from the split octonion algebra $Zorn(k)$ (see Section \ref{3}), which contains all composition algebras with dimensions $1,2$ and $4$ as subalgebras. We can also check that $(Sp_{6}\times GL_{1}, Z)$ is a PV with the $Sp_{6}$-invariant quartic homogeneous polynomial $J$ as relative invariant and the action of $GL_{1}$ is given by scalar multiplication (see Section \ref{5}, for details). We denote the set of all $k$-rational points of $(Sp_{6}\times GL_{1})$ by $(Sp_{6}\times GL_{1})(k)$. Let $Z^{ss}=\{z\in Z:J(z)\neq 0\} \subset Z$ be the set of all semi-stable points and $Z^{ss}_{k}$ be the set of all $k$-rational points in $Z^{ss}$. Then we can raise the following natural questions for the PV $(Sp_{6}\times GL_{1}, Z)$:

\begin{enumerate}
    \item What arithmetic objects are classified by the $(Sp_{6}\times GL_{1})(k)$-orbits in $Z^{ss}_{k}$?
    
    \item Do the orbits in $(Z^{ss}_{k})^{c}$ have some arithmetic interpretations? 
\end{enumerate}

\vspace{1 mm}

We prove that the $Sp_{6}(k)$-orbits in $J^{-1}\{0\}(k)-\{0\}\subset Z_{k}$ are in bijection with the isomorphism classes of composition algebras with dimensions $1,2$ and $4$, respectively; and the remaining orbits give us the composition algebras of dimension $8$, up to isomorphism. This can be expected from the construction of $Z$, as mentioned above. These results help us to get a description of the arithmetic objects parametrized by the $(Sp_{6}\times GL_{1})(k)$-orbits in $(Z^{ss}_{k})^{c}$ and $Z^{ss}_{k}$, respectively; answering the questions $(1)$ and $(2)$ above. In the end, we answer the same questions for the PV $(GSp_{6}\times GL_{1}, Z)$ (see Section \ref{5}), which leads to an alternate proof of Yukie's result in (\cite{YA1}, Theorem $1.10$, Theorem $4.11$). In addition, the above result, together with (\cite{PS}, Theorem $4.18$), reduces the basic structural study of composition algebras to the study of the $k$-rational orbit spaces of the representations $Z$ and $\wedge^{3}V_{6}$ of the group $Sp_{6}$, where $V_{6}$ is the standard $6$-dimensional representation of $Sp_{6}$.

\vspace{2 mm}

Let $(C,N_{C})$ be an octonion algebra defined over $k$ and $\mathcal{J}$ be the exceptional Jordan algebra $\mathcal{H}_{3}(C,I_{3})$. Let $GE_{6}$ be the group of all similitudes of the cubic norm defined on $\mathcal{J}$ (see \cite{KY}). Then $(GE_{6}\times GL_{2}, \mathcal{J}\oplus \mathcal{J})$ is a PV with a relative invariant of degree $12$ and the $k$-rational orbits in the semi-stable set classify the tuples $(\mathcal{J}^{\prime},L^{\prime})$ up to isomorphism, where $\mathcal{J}^{\prime}$ is an isotope of $\mathcal{J}$ and $L^{\prime}\subset \mathcal{J}^{\prime}$ a cubic \`etale subalgebra (see \cite{KY} for details). Here we will see a similar classification for the tuples $(C,F)$ up to isomorphism, where $C$ is an octonion over $k$ and $F\subset C$ a composition subalgebra of dimension $2$, in terms of the $k$-rational orbit space $Z^{ss}_{k}/(GSp_{6}\times GL_{1})(k)$.

\vspace{2 mm}

Again, we consider $(C,N_{C})$ to be an octonion algebra with $1_{C}\in C$ as the identity element. Let $V_{C}=1^{\perp}_{C}\subset C$ be the orthogonal complement of $1_{C}$ with respect to the quadratic norm $N_{C}$ and $\text{Aut}(C)$ be the group of all algebra automorphisms of $(C,N_{C})$, i.e., a simple connected group of type $G_{2}$. Then $V_{C}$ is invariant under the action of $\text{Aut}(C)$ and the $k$-rational orbits in $V_{C}(k)$ are given by the fibers of the restriction of the norm $N_{C}$ to $V_{C}$ (see Section \ref{5}). If $C$ is split, this result follows from (\cite{HJ}, Chapter I, Proposition $1.2$), over the number fields. We give a general proof for an arbitrary octonion $C$, over an arbitrary field k with characteristic different from $2$. We will also see that $(\text{Aut}(C)\times GL_{1},V_{C})$ is a PV (see Section \ref{5}) with the restriction of $N_{C}$ to $V_{C}$ as a relative invariant; and the $k$-rational orbit space in the semi-stable set is in bijection with the composition algebras of dimension $2$, if $C$ is split. These prehomogeneous vector spaces help us in producing a counterexample which shows that the $k$-rational orbit space in the complement of the semi-stable set of a PV may not be preserved under castling equivalence, if $C$ is a division algebra (see Remark \ref{CE}).

 \vspace{2 mm}

Now, any Freudenthal division algebra of dimension $9$ is of the form $\mathcal{H}(B,\sigma)$, where $B$ is a central division $K$-algebra of degree $3$ with an involution of the second kind $\sigma$ defined on it and the center $K$ is a quadratic extension of $k$ (see Section \ref{6}). Then $(NP(\mathcal{H}(B,\sigma))\times GL_{1},\mathcal{H}(B,\sigma))$ is a PV, where $NP(\mathcal{H}(B,\sigma))$ is the group of all isometries of the cubic norm defined on $\mathcal{H}(B,\sigma)$. We prove that the orbit space in the semi-stable set of this PV classifies all involutions of the second kind defined on $B$, and hence the isotopes of $\mathcal{H}(B,\sigma)$. In general, the PV $(NP(\mathcal{J})\times GL_{1},\mathcal{J})$ associated to a Freudenthal algebra $\mathcal{J}$ as above classifies the isotopes of $\mathcal{J}$ (up to isomorphism) in terms of the $k$-rational orbit space in the semi-stable set (see Section \ref{6}).

 \vspace{2 mm}

We can see that for the PV's, which we have discussed in Section \ref{5} (for example, $(GSp_{6}\times GL_{1},Z)$), the semi-stable set is a single orbit over the algebraic closure $\overline{k}$ and we have assigned an algebraic structure defined over $\overline{k}$, say $A$, to the semi-stable set. Then, the $k$-rational orbit space in the semi-stable set classifies the algebraic structures $B$ over $k$, which becomes isomorphic to $A$ after we change the base to $\overline{k}$. For $(GSp_{6}\times GL_{1},Z)$, we have assigned the pair of composition algebras $(Zorn(\overline{k}),\overline{k}\times \overline{k})$ defined over $\overline{k}$ to the semi-stable set $Z^{ss}$; and this is the only such pair consisting of an octonion and its quadratic subalgebra over $\overline{k}$. Here, the $k$-rational orbit space $Z^{ss}_{k}/(GSp_{6}\times GL_{1})(k)$ classifies the tuples $(C,F)$ consisting of an octonion $C$ and a quadratic subalgebra $F\subset C$ over the field $k$ such that $(C\otimes \overline{k},F\otimes \overline{k})\simeq (Zorn(\overline{k}),\overline{k}\times \overline{k})$. In other words, the $k$-rational orbit space in the semi-stable set classifies the $k$-forms of some algebraic structure defined over $\overline{k}$ (see \cite{SJP}, Chapter III, Section $1$, page 121). This gives an insight into the fact that the rational orbits in the semi-stable set of a prehomogeneous vector space parametrize several algebraic structures over the underlying field $k$.

 \vspace{2 mm}

 We start with some preliminaries in Section \ref{2}. In Section \ref{3}, we introduce the symplectic group $Sp_{6}$ and construct the irreducible representation of dimension $14$ for the same group using the algebra structure of the split octonion algebra. In Section \ref{4}, we prove that the $k$-rational orbit space of this representation classifies all composition algebras defined over the field $k$ (see Theorem \ref{M}). In Section \ref{5}, we show that the same phenomenon occurs for the orbit spaces in the same vector space of dimension $14$, under the actions of the groups $Sp_{6}\times GL_{1}$ and $GSp_{6}\times GL_{1}$ (see Theorem \ref{5.1}, \ref{5.2}, \ref{5.3}, \ref{5.4}, \ref{5.6}). We have also described the orbit spaces for some other PV's including $(\text{Aut}(C)\times GL_{1},1^{\perp}_{C})$, where $(C,N_{C})$ is any octonion algebra defined over $k$ and $1_{C}\in C$ is the identity. Finally, in Section \ref{6}, we have discussed what we get for the Freudenthal algebras from our parametrization of composition algebras, and described the PV's associated to the Freudenthal algebras. We have also computed the orbit spaces in detail for the PV's associated to the Freudenthal division algebras of dimension $9$. 
 

 \vspace{2 mm}

   Throughout this paper, we assume that the field $k$ has characteristic different from $2$ unless otherwise mentioned. By $X_{k}$ (or $X(k)$) we denote the set of all $k$-rational points of the object $X$. By orbits in a representation, we only consider the orbits of non-zero elements, which we call the \textit{non-zero orbits}. Let $(G,V)$ be a representation and $U$ be a $G$-orbit in $V$. We say that $G$ satisfies the Witt condition for the orbit $U$ if $G_{k}$ acts transitively on $U_{k}$. The Witt condition is satisfied for all orbits in the last two representations in Table \ref{1} (see \cite{IJ}, \cite{HJ}).
   

\section{Preliminaries}\label{2}

In this section, we define the notion of a prehomogeneous vector space and introduce some basics of composition algebras and quadratic forms over a field $k$ with $\text{char}(k)\neq 2$. We also discuss some properties of Freudenthal algebras defined over a field $k$ with $\text{char}(k)\neq 2,3$, and state some results on Galois cohomology that will be needed later. For a detailed study on these topics, we refer to (\cite{SK}), (\cite{IJ1}), (\cite{SV}), (\cite{LTY}), (\cite{KMRT}) and (\cite{SJP}).

\vspace{2 mm}

\noindent \textbf{Prehomogeneous vector space (PV):} Let $G$ be a connected reductive algebraic group defined over a field $k$ and $V$ be a rational representation of $G$ over $k$. Then $(G,V)$ is called a \textit{prehomogeneous vector space} (denoted by PV) if the following two conditions are satisfied:

\begin{enumerate}
    \item $\exists $ a homogeneous polynomial $f\in k[V]$ and a group homomorphism $\chi:G\rightarrow GL_{1}$ such that $f(g.x)=\chi(g)f(x), \forall g\in G, x\in V$;
    \item $V^{ss}=\{x\in V:f(x)\neq 0\}$ is a single $G$-orbit over $\overline{k}$.
\end{enumerate}

Here, the polynomial $f$ is called \textit{relative invariant} with respect to the \textit{character} $\chi$, and the set $V^{ss}=\{x\in V: f(x)\neq 0\}$ is called the set of \textit{semi-stable points} (or \textit{generic points}), which is a Zariski dense open orbit. If $x\in V^{ss}$ and $H=\text{Stab}_{G}(x)$, then $\text{dim}(G)-\text{dim}(H)=\text{dim}(V)$. In (\cite{SK}), the above definition of PV is the definition of regular PV's with reductive groups (see \cite{SK}, Remark $26$, page $73$). However, we will consider the above definition for a PV, following the notion in (\cite{IJ1}). We say that a PV $(G,V)$ is \textit{irreducible} if $V$ is irreducible as a rational $G$-representation. 

\vspace{1 mm}

Let $G\subset GL_{n}$ be a connected reductive algebraic group defined over the field $k$, for some positive integer $n$. Let $G_{1}=G\times GL_{p}$ and $G_{2}=G\times GL_{q}$, where $p+q=n$. We have a natural action of $G_{1}$ and $G_{2}$ on the vector spaces $W_{1}=V_{n}\otimes V_{p}$ and $W_{2}=V^{\ast}_{n}\otimes V_{q}$, respectively; where $V_{n},V_{p},V_{q}$ are vector spaces with dimensions $n,p,q$, respectively, and $V^{\ast}_{n}$ denotes the dual of $V_{n}$. Then we have the following proposition from (\cite{IJ1}, Proposition $3.1$), which is a refinement of (\cite{SK}, page $37$-$38$, Proposition $7,9$; page $73$, Remark $26$).

\begin{proposition}[Igusa, 1987]
    $(G_{1},W_{1})$ is a PV if and only if $(G_{2},W_{2})$ is a PV. If $W^{ss}_{1}\subset W_{1}$ and $W^{ss}_{2}\subset W_{2}$ are the sets of semi-stable points, the $k$-rational orbit spaces in these two sets are in bijection, i.e., $W^{ss}_{1}(k)/G_{1}(k)\longleftrightarrow W^{ss}_{2}(k)/G_{2}(k)$.
\end{proposition}

So, due to the above result, if $(G_{1},W_{1})$ is a PV and we can determine the orbit space in the semi-stable set for $(G_{1},W_{1})$, the same gets determined for the PV $(G_{2},W_{2})$. We say that two PV's $(G_{1},W_{1})$ and $(G_{2},W_{2})$ in the above form are \textit{castling transforms} of each other, and we write it as $(G_{1},W_{1})\sim_{CT}(G_{2},W_{2})$. In general, if $(G,V)$ and $(H,W)$ are two PV's such that there exists a finite family of PV's $\{(G_{i},W_{i})\}^{n}_{i=1}$ with $$(G,V)\sim_{CT} (G_{1},W_{1}), \text{ } (G_{n},W_{n})\sim_{CT} (H,W), \text{ } (G_{i},W_{i})\sim_{CT} (G_{i+1},W_{i+1})$$ for $i=1,..,n-1$, we say $(G,V)$ is \textit{castling equivalent} to $(H,W)$. This is an equivalence relation on the family of PV's. So, it is enough to determine the orbit space in the semi-stable set for one representative from an equivalence class. But, for the complement of the semi-stable set, the orbit space may not be preserved under castling equivalence. We will see such an example in Section \ref{5} (see Remark \ref{CE}). For more details on this equivalence relation, we refer to (\cite{IJ1}, page 270-271) and (\cite{SK}, Section $2,4$).

 \begin{example} We consider $G=GL_{n}$ defined over the field $k$ and $V=\text{Sym}_{n}$, the space of all symmetric $n\times n$ matrices. We have an action of $G$ on $V$ defined by $g.x=gxg^{t}$, for $g\in G, x\in V$. With this action $(G,V)$ is a PV with the determinant as a relative invariant with respect to the character $\chi(g)=\text{det}(g)^{2}, \text{ for } g\in G$. The set of semi-stable points $V^{ss}$ is the set of all non-singular symmetric matrices, which is a single orbit over $\overline{k}$ (the algebraic closure of $k$). The stabilizer of $I_{n}\in V^{ss}$ is the orthogonal group $O_{n}=\{g\in G:gg^{t}=I_{n}\}$. We have $\text{dim}(O_{n})=n(n-1)/2$. So, $\text{dim}(G)-\text{dim}(O_{n})=\text{dim}(V)$. Here, $V^{ss}_{k}/G_{k}$ classifies the non-degenerate quadratic forms of rank $n$ up to isometry and $(V^{ss}_{k})^{c}/G_{k}$ classifies the isometry classes of quadratic forms with rank less than or equal to $n-1$. So in this example, the orbits in the complement of the semi-stable set classify similar algebraic structures, as done by the orbit space in the semi-stable set.
\end{example}

\vspace{2 mm}

\noindent\textbf{Composition algebras:} Let $k$ be a field with $\text{char}(k)\neq 2$. A \textit{composition algebra} $C$ over $k$ is a unital (not necessarily associative) algebra with a non-degenerate quadratic form $N_{C}$ defined on it such that $N_{C}(xy)=N_{C}(x)N_{C}(y)$, $\forall x,y\in C$. We call the quadratic form $N_{C}$ the \textit{norm} of the composition algebra and write the algebra as $(C,N_{C})$. We can consider the field $k$ itself as a composition algebra $C$ with the quadratic norm defined by $N_{C}(x)=x^{2}$, $x\in k$. We denote the bilinear form associated to the quadratic form $N_{C}$ by $b_{N_{C}}$, i.e., $b_{N_{C}}(x,y)=N_{C}(x+y)-N_{C}(x)-N_{C}(y),$ for $ x,y\in C$; and we define a map $C\rightarrow C$, called the \textit{conjugation involution}, by $x\mapsto \overline{x}=b_{N_{C}}(x,1_{C})1_{C}-x$ for $x\in C$ and $1_{C}\in C$ is the identity element. We can prove that $N_{C}(x)=x\overline{x}, \forall x\in C$ (see \cite{SV}, Chapter $1$).

\vspace{1 mm}

Let $C$ be a composition algebra over $k$ and $D\subset C$ be a finite-dimensional composition subalgebra such that $D\neq C$. We can choose $a\in D^{\perp}$ such that $N_{C}(a)\neq 0$ and let $\lambda=-N_{C}(a)$. Then $D_{1}=D\bigoplus Da$ is a composition subalgebra of $C$, where the product, norm and conjugation are defined by the following formulas, respectively (see \cite{SV}, Chapter $1$, Proposition $1.5.1$).

\begin{center}
    $(x+ya).(u+va)=(xu+\lambda \overline{v}y)+(vx+y\overline{u})a$, for $x,y,u,v\in D;$\\
    \vspace{1 mm}
    $N_{D_{1}}(x+ya)=N_{C}(x)-\lambda N_{C}(y)$, for $x,y\in D;$\\
    \vspace{1 mm}
    $\overline{(x+ya)}=\overline{x}-ya$, for $x,y\in D$.
\end{center}

Now, we consider a composition algebra $D$ with the norm $N_{D}$ over $k$ and a scalar $\lambda\in k^{\times}=k-\{0\}$. We define a multiplication on $C=D\bigoplus D$ by $(x,y).(u,v)=(xu+\lambda \overline{v}y,vx+y\overline{u})$ and a quadratic norm $N_{C}$ by $N_{C}((x,y))=N_{D}(x)-\lambda N_{D}(y)$, where $x,y,u,v\in D$. Then $C$ is a composition algebra with norm $N_{C}$ if $D$ is associative, and $C$ is associative if and only if $D$ is commutative and associative (see \cite{SV}, Chapter $1$, Proposition $1.5.3$). By this process, we can generate a composition algebra with double dimension from a given associative composition algebra. This process is called \textit{Cayley-Dickson doubling}. All composition algebras over the field $k$ can be obtained by repeated Cayley-Dickson doubling starting from the field $k$ itself, and the composition algebras of dimension $8$ are non-associative (see \cite{SV}, Chapter $1$, Theorem $1.6.2$).

\vspace{1 mm}

Composition algebras exist only in dimensions $1,2,4$ and $8$ (see \cite{SV}, Chapter 1 for details). A composition algebra of dimension $2$ is either a quadratic field extension over the field $k$ or isomorphic to $k\bigoplus k$. The composition algebras with dimensions $4$ and $8$ are called \textit{quaternion} and \textit{octonion algebras}, respectively.

\vspace{1 mm}

We define an \textit{isomorphism} between two composition algebras as an isomorphism of the underlying algebras. An isomorphism between two composition algebras gives us an isometry between the quadratic norms. If there is an isometry between the norms of two composition algebras, they are isomorphic (see \cite{SV}, Chapter $1$).

\vspace{1 mm}

If the norm $N_{C}$ on a composition algebra $C$ over $k$ is isotropic, i.e., $\exists $ $ x\in C - \{0\}$ such that $N_{C}(x)=0$, we call it \textit{split composition algebra}. If this is not the case, then $\forall x\in C-\{0\}, N_{C}(x)\neq 0$ and we can see that $N_{C}(x)^{-1}\overline{x}$ is an inverse of $x$. Such an algebra $C$ is called \textit{composition division algebra}. Split composition algebras exist in dimensions $2,4$ and $8$, and are unique up to isomorphism.

\vspace{3 mm}

\noindent \textbf{Pfister forms:} For $n$ elements $a_{1},a_{2},..,a_{n}\in k^{\times}$, we define an $n$-fold \textit{Pfister form} to be the $2^{n}$ dimensional quadratic form $\bigotimes_{i=1}^{n} \langle 1,-a_{i}\rangle$ over $k$. We also denote it by $\langle\langle a_{1},a_{2},.., a_{n}\rangle\rangle$. Since any $n$-fold Pfister form, say $q$, represents $1$, we have $q\simeq\langle 1\rangle\perp q^{\prime}$; where $q^{\prime}$ is called the \textit{pure subform} of $q$. By Witt's cancellation, isometry type of $q^{\prime}$ is uniquely determined by that of $q$. A Pfister form is isotropic if and only if it is hyperbolic. The norm of a composition algebra defined over a field $k$, is always a Pfister form over $k$. For more details on Pfister forms, we refer to (\cite{LTY}, Chapter X).

\vspace{3 mm}

\noindent \textbf{Discriminant:} Let $(V,q)$ be a non-degenerate quadratic space of dimension $n$ over the field $k$ and $(q_{ij})$ be the matrix representative of $q$ with respect to some fixed basis. The \textit{discriminant} of $q$ is defined by \[\text{disc}(q)=(-1)^{n(n-1)/2}\text{det}((q_{ij})).{k^{\times}}^{2}\in k^{\times}/{k^{\times}}^{2},\]

\noindent (see \cite{KMRT}, Conventions and Notations). We say $q$ is a \textit{trivial discriminant} quadratic form over $k$ if $\text{det}((q_{ij}))\in k^{\times 2}$. If $q$ has rank $3$, this is equivalent to $\text{disc}(q)=(-1).{k^{\times}}^{2}$.

\vspace{1 mm}

 Let $K/k$ be a quadratic extension and $(V,h)$ be a non-degenerate $K/k$-hermitian space of rank $n$ over $K$ with respect to the non-trivial involution on $K$. Then the \textit{discriminant} of $h$ is defined by \[\text{disc}(h)=(-1)^{n(n-1)/2}\text{det}((h_{ij})). N(K/k)\in k^{\times}/N(K/k)\]

\noindent (see \cite{KMRT}, Chapter II), where $(h_{ij})$ is the hermitian matrix representing $h$ with respect to some fixed basis of $V$ and $N(K/k)\subset k^{\times}$ is the group of all norm values of non-zero elements in $K$. We say $h$ is a \textit{trivial discriminant} $K/k$-hermitian form if $\text{det}((h_{ij}))$ is a $K/k$-norm. If $h$ has rank $3$, this is equivalent to $\text{disc}(h)=(-1).N(K/k)$.

\vspace{3 mm}

\noindent\textbf{Zorn algebra:} On the $k$-vector space $k^{3}$ we have an alternating bilinear map $\times : k^{3}\times k^{3}\rightarrow k^{3}$ defined by $x\times y=(x_{2}y_{3}-x_{3}y_{2}, x_{3}y_{1}-x_{1}y_{3}, x_{1}y_{2}- x_{2}y_{1})^{t}$, where $x=(x_{1},x_{2},x_{3})^{t},y=(y_{1},y_{2},y_{3})^{t}\in k^{3}$. Let us consider the vector space of dimension $8$, \[Zorn(k)=\left(\begin{array}{cc}
k & k^{3} \\
k^{3} & k
\end{array}\right).\]

\noindent We can define an algebra structure on the above space with the multiplication $Zorn(k)\times Zorn(k)\rightarrow Zorn(k)$ defined by

  \[\left(\begin{array}{cc}
  a & x \\
  y & b
  \end{array}\right). 
  \left(\begin{array}{cc}
  a^{\prime} & x^{\prime} \\
  y^{\prime} & b^{\prime}
  \end{array}\right) 
  =
  \left(\begin{array}{cc}
  aa^{\prime}+x^{t}y^{\prime} & ax^{\prime}+b^{\prime}x+y\times y^{\prime} \\
  a^{\prime}y+by^{\prime}+x\times x^{\prime} & bb^{\prime}+y^{t}x^{\prime} 
  \end{array}\right),\]

\noindent where $a,b,a^{\prime},b^{\prime}\in k$ and $x,y,x^{\prime},y^{\prime}\in k^{3}$. This is a non-associative $k$-algebra with $\left(\begin{array}{cc}
1 & 0 \\
0 & 1
\end{array}\right)$
as the identity element. This algebra is called the \textit{Zorn algebra} of vector matrices and is isomorphic to the split octonion algebra obtained by doubling the split quaternion algebra $M_{2}(k)$ with determinant as quadratic norm. The norm on $Zorn(k)$ is defined by \[N(\left(\begin{array}{cc}
  a & x \\
  y & b
  \end{array}\right))= ab-x^{t}y, \text{ where } a,b\in k \text{ and } x,y \in k^{3}. \]

\vspace{3 mm}

\noindent \textbf{Octonion algebras associated to rank $3$ hermitian forms:} Let $K/k$ be a quadratic field extension with the standard involution defined on it. We consider $(V,h)$ as a non-degenerate $K/k$-hermitian space of rank $3$ with trivial discriminant. Take $C=K\bigoplus V$, a vector space of dimension $8$ over $k$.  We define a quadratic form on $C$ by $N_{C}(a,x)=N(a)+h(x,x)$, where $ a \in K, x\in V$ and $N$ is the $K/k$-norm. Then we can define a multiplication on $C$, with respect to which $C$ forms an octonion algebra over $k$ with $N_{C}$ as the multiplicative norm (see \cite{JN}, Theorem $3$) and $C\bigotimes_{k} K\simeq Zorn(K)$. The algebra $(C,N_{C})$ contains $K$ as a subalgebra. If we take $K$ to be the split quadratic algebra $k\times k$ over $k$, we can have the same construction for the split octonion algebra with $(V,h)$ as a non-degenerate $K/k$-hermitian space of rank $3$. We refer to (\cite{JN}) and (\cite{TML}) for more details.

\vspace{3 mm}

\noindent \textbf{Freudenthal algebras:} Let us assume $\text{char}(k)\neq 2,3$ to describe these algebras. Let $C$ be any composition algebra over the field $k$ and the map $x\mapsto \overline{x}$ for $ x\in C$ be the conjugation involution on $C$. Let $M_{3}(C)$ be the matrix algebra of $3\times 3$ matrices with entries in $C$ and $\Gamma = \text{diag}(\gamma_{1},\gamma_{2},\gamma_{3})\in GL_{3}(k)$. We consider the involution on $M_{3}(C)$ defined by $X^{*\Gamma}:=\Gamma^{-1}\overline{X}^{t}\Gamma$ for $ X\in M_{3}(C)$, where $\overline{X}=(\overline{x_{ij}})$ if $X=(x_{ij})$. Let $\mathcal{H}_{3}(C,\Gamma)=\{X\in M_{3}(C):X^{*\Gamma}=X\}\subset M_{3}(C)$ and we define a multiplication on this subspace by $$X.Y=\frac{1}{2}(XY+YX),\text{ for } X,Y\in \mathcal{H}_{3}(C,\Gamma),$$ where $XY$ denotes the usual matrix multiplication. The algebras $\mathcal{H}_{3}(C,\Gamma)$ with respect to the above multiplication are called \textit{reduced Freudenthal algebras}. These algebras contain zero divisors and $dim(\mathcal{H}_{3}(C,\Gamma))=3(\text{dim}(C)+1)$. Let $\mathcal{A}$ be an algebra over $k$ such that for some field extension $K/k$, $\mathcal{A}\otimes K$ is isomorphic to a reduced Freudenthal algebra defined over $K$. Then $\mathcal{A}$ is called a \textit{Freudenthal algebra}. Freudenthal algebras are cubic Jordan algebras and every element $a$ in a Freudenthal algebra $\mathcal{A}$ satisfies a cubic monic polynomial over $k$, of the form $P(x)=x^{3}-T_{\mathcal{A}}(a)x^{2}+S_{\mathcal{A}}(a)x-N_{\mathcal{A}}(a)$ (see \cite{KMRT}, Chapter IX). Here, $a\mapsto T_{\mathcal{A}}(a)$ gives us a linear map $T_{\mathcal{A}}:\mathcal{A}\rightarrow k$, called the \textit{trace map}. We can define a bilinear map $T:\mathcal{A}\times \mathcal{A}\rightarrow k$ by $T(a,b)=T_{\mathcal{A}}(a.b),$ for $a,b\in \mathcal{A}$; which is called the \textit{bilinear trace form}. If we let $a^{\#} :=a^{2}-T_{\mathcal{A}}(a)a+S_{\mathcal{A}}(a)$, then $a\mapsto a^{\#}$ defines a quadratic map $\#:\mathcal{A}\rightarrow \mathcal{A}$ satisfying $aa^{\#}=N_{\mathcal{A}}(a), \forall a\in \mathcal{A}$. This map is called the \textit{Freudenthal adjoint} and let $\times : \mathcal{A}\times \mathcal{A}\rightarrow \mathcal{A}$ be the associated bilinear map defined by $a\times b=(a+b)^{\#}-a^{\#}-b^{\#},$ for $ a,b\in \mathcal{A}$; which is called the \textit{Freudenthal cross product}. The scalar $N_{\mathcal{A}}(a)$ is called the \textit{generic norm} of $a\in \mathcal{A}$. For more details on these algebras, we refer the reader to (\cite{GPR}) and (\cite{KMRT}).

\vspace{3 mm}

\noindent \textbf{Exact sequence in Galois cohomology:} Let $G$ be a reductive algebraic group defined over the field $k$ and $H\subset G$ be a smooth subgroup of $G$. Then we have the following Proposition and Corollary from (\cite{SJP}, Chapter I, Section $5$, page 50, Proposition $36$, Corollary $1$).

\begin{proposition}\label{ESC}
    The sequence of cohomology sets \[1 \rightarrow H_{k}\rightarrow G_{k}\rightarrow (G/H)_{k}\rightarrow H^{1}(k,H)\rightarrow H^{1}(k,G), \] is exact.
\end{proposition}

\begin{corollary}\label{ESCC}
    The kernel of the map $H^{1}(k,H)\rightarrow H^{1}(k, G)$ may be identified with the quotient space of $(G/H)_{k}$ by the action of the group $G_{k}$.
\end{corollary}

As we are using the Galois cohomology here, it is necessary to assume that $H$ is a smooth subgroup of $G$ in the above discussion (see \cite{BG}, Section $2$, Proposition $1$, Remark $2$). We will see later that these two results help us to get cohomological interpretations for the orbit spaces in many representations, where the stabilizers of the orbit representatives are smooth algebraic groups (see Section \ref{4}, \ref{5}). For more details on Galois cohomology, we refer the reader to (\cite{SJP}) and (\cite{KMRT}, Chapter VII).

\section{The symplectic group and an irreducible representation}\label{3}

In this section, we discuss the symplectic group $Sp_{6}$ of $6\times 6$ matrices from $GL_{6}$ defined over the field $k$; and construct the $14$-dimensional irreducible representation $X$ of $Sp_{6}$, which we identify as a vector space with the structurable algebra $Z$ mentioned in the Introduction (see Section \ref{I}). From now on, we write $G$ for the group $Sp_{6}$. We explain the action of $G$ on $X$ and describe the quartic $G$-invariant defined on $X$, which generates the ring of invariants.
\vspace{ 1 mm}

Let $V$ be the $6$-dimensional vector space which we write as \[V= \left(\begin{array}{cc}
0 & V_{3} \\
V_{3} & 0
\end{array}\right),\] the off-diagonal subspace of the Zorn algebra and $V_{3}$ is the $3$-dimensional vector space $k^{3}$.

\vspace{2 mm}

Let $\mathcal{B}=\{e_{1},e_{2},..,e_{6}\}$ be the basis of $V$, where

 \begin{center}
     $e_{i}=\left(\begin{array}{cc}
0 & f_{i} \\
0 & 0
\end{array}\right)$, $e_{i+3}=\left(\begin{array}{cc}
0 & 0 \\
f_{i} & 0
\end{array}\right)$, for $i=1,2,3$
 \end{center}
 
\noindent and $\{f_{1},f_{2},f_{3}\}$ is the standard basis of $V_{3}$.

\vspace{3 mm}

\noindent \textbf{The symplectic group:} Up to isometry there exists a unique non-degenerate alternating bilinear form on $V$, say $Q$, which is represented by the alternating matrix \[ M_{Q}= \left(\begin{array}{cc}
0 & I_{3} \\
-I_{3} & 0
\end{array}\right), \] with respect to the basis $\mathcal{B}$. Then the symplectic group $Sp_{6}$ is the group $Sp_{6}=\{g\in GL_{6}: gM_{Q}g^{t}=M_{Q}\}$. It follows that the group $G=Sp_{6}$ can be described as \[G=\biggl\{ \left(\begin{array}{cc}
\alpha & \beta \\
\gamma & \delta
\end{array}\right) \in GL_{6}: \alpha,\beta,\gamma,\delta \text{ are $3\times 3$ matrices; }\alpha\beta^{t}=\beta\alpha^{t}, \gamma\delta^{t}=\delta\gamma^{t}, \alpha\delta^{t}-\beta\gamma^{t}=I_{3} \biggl\}.\]

\vspace{1 mm}

\subsection{The irreducible representation} Consider the $6$-dimensional vector space $V$ as above. We know that $SL_{3}\subset \text{Aut}(Zorn(k))$ (the group of all algebra automorphisms) acts on the Zorn algebra by stabilizing the off-diagonal subspace $V$ and the restriction of this action to $V$ is given by \[A. \left(\begin{array}{cc}
0 & u \\
v & 0
\end{array}\right) = \left(\begin{array}{cc}
0 & Au \\
(A^{t})^{-1}v & 0
\end{array}\right),\text{ for } A\in SL_{3} \text{ and } u,v\in V_{3},\] (see \cite{JN}, Theorem $3,4$ for details). The $G$-action on $V$ is given by

    \[g. \left(\begin{array}{cc}
0 & u \\
v & 0
\end{array}\right) = \left(\begin{array}{cc}
0 & \alpha u+\beta v \\
\gamma u + \delta v & 0
\end{array}\right), \text{ for } g =\left(\begin{array}{cc}
\alpha & \beta \\
\gamma & \delta
\end{array}\right)\in G \text{ and } u,v \in V_{3};\]

\noindent which is induced by the $GL(V)$-action on $V$, as $G=Sp_{6}\subset GL(V)$.

\vspace{2 mm}

Again, we have the following subgroup of $G$, which is isomorphic to $SL_{3}$. $$SL_{3}\simeq \biggl\{ \left(\begin{array}{cc}
A & 0 \\
0 & (A^{t})^{-1}
\end{array}\right): A\in SL_{3} \biggl\}\subset G$$ If we take the restriction of the $G$-action (defined above) to this subgroup $SL_{3}\subset G$, we get an identification of the algebra automorphisms of $Zorn(k)$ contained in $SL_{3}$ (which stabilizes $V\subset Zorn(k)$) as a subgroup inside $G\subset GL(V)$.

\vspace{2 mm}

Let us consider the alternating trilinear form $T$ on $V$ defined by \[T\biggl(\left(\begin{array}{cc}
0 & u_{1} \\
v_{1} & 0
\end{array}\right),\left(\begin{array}{cc}
0 & u_{2} \\
v_{2} & 0
\end{array}\right),\left(\begin{array}{cc}
0 & u_{3} \\
v_{3} & 0
\end{array}\right)\biggl)=(v_{1}\times v_{2})^{t}v_{3},\] where $u_{i},v_{i}\in V_{3},\text{ for } i=1,2,3$ and `$\times$' denotes the usual cross product on $V_{3}$.

\begin{proposition}
    The alternating trilinear form $T$ is $SL_{3}$-invariant.
\end{proposition}

\begin{proof}

We have mentioned earlier that $SL_{3}$ is the subgroup of the group of all algebra automorphisms of $Zorn(k)$, which maps the off-diagonal subspace $V$ to itself. Let us consider \[\left(\begin{array}{cc}
0 & u_{1} \\
v_{1} & 0
\end{array}\right),\left(\begin{array}{cc}
0 & u_{2} \\
v_{2} & 0
\end{array}\right) \in V \text{ and } A\in SL_{3}.\] Since $A$ is an algebra automorphism of $Zorn(k)$ and it fixes the diagonal elements (see \cite{JN}, Theorem $4$), we have \[A.\biggl(\left(\begin{array}{cc}
0 & u_{1} \\
v_{1} & 0
\end{array}\right).\left(\begin{array}{cc}
0 & u_{2} \\
v_{2} & 0
\end{array}\right)\biggl)=\left(\begin{array}{cc}
0 & Au_{1} \\
(A^{t})^{-1}v_{1} & 0
\end{array}\right).\left(\begin{array}{cc}
0 & Au_{2} \\
(A^{t})^{-1}v_{2} & 0
\end{array}\right).\] Comparing both sides, we get $$(A^{t})^{-1}v_{1}\times (A^{t})^{-1}v_{2}=A(v_{1}\times v_{2}) \text{ and } (Au_{1})^{t}((A^{t})^{-1}v_{2})=u^{t}_{1}v_{2},$$ where $u_{1},u_{2},v_{1},v_{2}\in V_{3}$ and $A\in SL_{3}$ are arbitrary. Using the above relations, we have $$T(A.U_{1},A.U_{2},A.U_{3})=T(U_{1},U_{2},U_{3}),$$ where $A\in SL_{3}$ and $U_{r}\in V$ for $r=1,2,3$. So, we get the $SL_{3}$-invariance of $T$.

\end{proof}

The space of all alternating trilinear forms on $V$ can be identified with $\wedge^{3}V \simeq \wedge^{3}V^{\ast}$ using the non-degenerate alternating bilinear form $Q$ on $V$, and this isomorphism is induced by the isomorphism of $V$ with its dual $V^{\ast}$ given by the map $e_{i}\mapsto e_{i}^{\ast}$ where \[ e_{i}^{\ast}(x)=Q(e_{i},x) \text{ for }x\in V, i=1,2,...,6;\] with respect to the basis $\mathcal{B}=\{e_{1},e_{2},.., e_{6}\}$. By computing the alternating trilinear form given by $e_{1}^{\ast}\wedge e_{2}^{\ast}\wedge e_{3}^{\ast}$ on $V$, we can see that $T=e_{1}^{\ast}\wedge e_{2}^{\ast}\wedge e_{3}^{\ast}$. So, the element in $\wedge^{3}V$ corresponding to $T$ is given by $e_{1}\wedge e_{2}\wedge e_{3}$ under the isomorphism defined by $Q$. Now we consider the $G$-action on $\wedge^{3}V$ induced by the $G$-action on $V$. Then $e_{1}\wedge e_{2}\wedge e_{3}$ is fixed by the above copy of $SL_{3}\subset G$, which contains algebra automorphisms of $Zorn(k)$ stabilizing the off-diagonal subspace $V$. Let $X=\text{span}\{G.(e_{1}\wedge e_{2}\wedge e_{3})\}$. From now on, we will write $e_{i}e_{j}e_{l}$ for $e_{i}\wedge e_{j}\wedge e_{l}$.

\begin{proposition}
    The vector space $X$ is an irreducible representation of $G$.
\end{proposition}

\begin{proof}
    We have $X=\text{span}\{G.(e_{1}e_{2}e_{3})\}\subset \wedge^{3}V$. It is straightforward to check that $e_{1}e_{2}e_{3}$ is a highest weight vector with respect to the maximal diagonal subgroup in $G$, and from the arguments in (\cite{YA1}, Section $2$, page $41$) it follows that $X$ is irreducible as a representation of $G=Sp_{6}$. For more details, we refer the reader to (\cite{YA1}, Section $1$, page $35$-$36$; Section $2$, page $41$).
    
\end{proof}

\begin{remark}
We can start with any $SL_{3}$-invariant alternating trilinear form instead of $T$ on the off-diagonal subspace $V\subset Zorn(k)$, and proceed similarly to get the same $X$.
\end{remark}

\vspace{1 mm}

\noindent\textbf{Alternative description of $X$:} We have a natural contraction map $\phi : \wedge^{3}V\rightarrow V$ defined by $\phi (v_{1}\wedge v_{2}\wedge v_{3})=Q(v_{2},v_{3})v_{1}-Q(v_{1},v_{3})v_{2}+Q(v_{1},v_{2})v_{3}$, where $Q$ is the standard non-degenerate alternating bilinear form defined on $V$ (see \cite{FH}, Lecture $17$). Then $\phi$ is surjective. We denote the subspace $\text{ker}(\phi)\subset \wedge^{3}V$ by $\wedge_{0}^{3}V$ or $\wedge^{3}_{0}V_{6}$ (as $\text{dim}(V)=6$), the space mentioned in the Introduction (Table \ref{1}).

\begin{proposition}
    $\wedge^{3}_{0}V$ is the irreducible $Sp_{6}$-representation $X$.
\end{proposition}

\begin{proof}

Clearly $e_{1}e_{2}e_{3}\in \wedge_{0}^{3}V$, and therefore $X\subset \wedge_{0}^{3}V$. Since $\text{dim}(X)=14$ (see \cite{YA1}, Section 1,2) and also $\text{dim}(\wedge_{0}^{3}V)=14$, we have $X=\wedge^{3}_{0}V$.

\end{proof}

\vspace{1 mm}

\noindent\textbf{The identification of $X$ with $Z$:} We have $\wedge^{3}V=\text{span} \{ e_{i} e_{j} e_{l}:1\leq i<j<l\leq 6\}$, a $20$-dimensional vector space. For $x\in \wedge^{3}V$, we denote the coordinates by $x_{ijl}$ with respect to the above basis. We can choose an ordering of the above basis for $\wedge^{3}V$ such that we can identify $X\subset\wedge^{3} V$ in the following way: $$ X=\{(-x_{0},-y_{0},A,B):x_{0},y_{0} \text{ are scalars and } A,B \text{ are symmetric matrices of order }3\},$$ $$\wedge^{3}V=\{(-x_{0},-y_{0},A,B):x_{0},y_{0} \text{ are scalars and } A,B \text{ are }3\times 3 \text{ matrices}\},$$ where $x_{0}=-x_{123}, y_{0}=-x_{456}$ and the matrices $A,B$ are given by
\begin{center}
    
$A=(a_{ij})=\left(\begin{array}{ccc}
x_{423} & x_{143} & x_{124} \\
x_{523} & x_{153} & x_{125} \\
x_{623} & x_{163} & x_{126}
\end{array}\right)$ and $B=(b_{ij})=\left(\begin{array}{ccc}
x_{156} & x_{416} & x_{451} \\
x_{256} & x_{426} & x_{452} \\
x_{356} & x_{436} & x_{453}
\end{array}\right).$

\end{center}

\vspace{1 mm}

Let $\text{char}(k)\neq 2,3$, and we consider the Freudenthal algebra $\mathcal{H}_{3}(k)=\mathcal{H}_{3}(k,I_{3})$. We have \[Z=\left(\begin{array}{cc}
k & \mathcal{H}_{3}(k) \\
\mathcal{H}_{3}(k) & k
\end{array}\right)\] and we get an algebra structure on $Z$ by defining a multiplication on it by 
\[\left(\begin{array}{cc}
a_{1} & A_{1} \\
B_{1} & b_{1}
\end{array}\right). \left(\begin{array}{cc}
a_{2} & A_{2} \\
B_{2} & b_{2}
\end{array}\right) = \left(\begin{array}{cc}
a_{1} a_{2} + T(A_{1}, B_{2}) & a_{1}A_{2}+b_{2}A_{1}+B_{1}\times B_{2} \\
a_{2}B_{1}+b_{1}B_{2}+A_{1}\times A_{2} & b_{1}b_{2}+T(B_{1},A_{2})
\end{array}\right),\]

\noindent where $a_{j},b_{j}\in k$ and $ A_{j},B_{j}\in \mathcal{H}_{3}(k),$ for $ j=1,2$; `$\times$' and $T$ denote the Freudenthal cross product and the bilinear trace form defined on $\mathcal{H}_{3}(k)$, respectively (see \cite{AB} for details). As a vector space (also as a $G$-representation) we can identify $X$ with $Z$ from the above description of $X$ by assigning $(-x_{0},-y_{0},A,B)\in X$ to $\left(\begin{array}{cc}
-x_{0} & A \\
B & -y_{0}
\end{array}\right) \in Z$. We can consider this identification between two $Sp_{6}$-representations in the case when $\text{char}(k)=3$ as well.

\subsection{The quartic invariant} The irreducible representation $X$ has a $1$-dimensional ring of invariants generated by a quartic homogeneous polynomial $J$ defined over $k$, which is uniquely determined up to a scalar multiple, i.e., $k[X]^{Sp_{6}}=k[J]$; since the $Sp_{6}$-orbit in $X$ with maximal dimension has dimension $13$, i.e., it has codimension $1$ in $X$ (see \cite{IJ}, Proposition $7$; \cite{PVL}, Proposition $12$; \cite{BGL}, Proposition $6.1$). To describe $J$, we consider $\wedge^{3}V$ as an irreducible representation of $SL_{6}$. This representation has a quartic invariant, say $J_{1}$, which generates the ring of all $SL_{6}$-invariant polynomials defined on $\wedge^{3}V$ (see \cite{BGL}, Section $9$). Following the calculations in (\cite{BGL}, Section $9$), we take $V=V_{2}\oplus V_{4}$; where $V_{2}=\text{span}\{e_{1},e_{4}\}$, $V_{4}=\text{span}\{e_{2},e_{3},e_{5},e_{6}\}$ and $\mathcal{B}=\{e_{1},e_{2},...,e_{6}\}$ is the basis of $V$ as we have mentioned earlier. Then there is an inclusion $w:V_{2}\otimes \wedge^{2}V_{4}\rightarrow \wedge^{3}V$ defined by $c\otimes x\mapsto c\wedge x$, $c\in V_{2},x\in \wedge^{2}V_{4}$ and in the image of $w$, $J_{1}$ can be expressed as $$J_{1}(e_{1}\wedge x +e_{4}\wedge y)=<x,y>^{2}-4\text{Pf}(x)\text{Pf}(y),$$ where $x,y\in \wedge^{2}V_{4}$ and $<,>$ denotes the polar of the quadratic form $\text{Pf}$, the Pfaffian (see \cite{BGL}, ($9.4$); \cite{SK}, page $83$). Since $Sp_{6}\subset SL_{6}$, $J_{1}$ is also $Sp_{6}$-invariant. Before proceeding further, we first prove the following proposition which follows from (\cite{IJ}).

\begin{proposition}\label{P3.4}
    Any element $x\in X$ can be made $G$-equivalent to an element of the form $$y=-x_{0}e_{1}e_{2}e_{3}-y_{0}e_{4}e_{5}e_{6}+y_{1}e_{1}e_{5}e_{6}+y_{2}e_{4}e_{2}e_{6}+y_{3}e_{4}e_{5}e_{3}$$ over $k(x)$, i.e., $y=(-x_{0},-y_{0},A,B)\in X$ where $A=0$ and $B=\text{diag}(y_{1},y_{2},y_{3})$.
\end{proposition}

\begin{proof}

We can easily check that any $x\in X$ can be made $G$-equivalent over $k(x)$ to some $y\in X$ in which the coefficient of $e_{1}e_{2}e_{3}$ is non-zero, as $X=\text{span}\{G.(e_{1}e_{2}e_{3})\}$. So, we may assume that $x_{0}\neq 0$ for $x=(-x_{0},-y_{0},A,B)\in X $. If we take the element $g\in G$, with $\alpha =\delta =I_{3}$, $\beta =0$ and $\gamma =x_{0}^{-1}A$, we get $g.x=(-x_{0},-y_{0}^{\prime},0,B^{\prime})$, for some scalar $y_{0}^{\prime}$ and $B^{\prime}$ is some symmetric matrix of order $3$. Again, for an element $g^{\prime}\in G$ with $\beta=\gamma =0$, the transformation formula for $x=(-x_{0},-y_{0},A,B)\in X$ under the action of $g^{\prime}$ is given by

\begin{center}
     $x_{0}\longrightarrow \text{det}(\alpha)x_{0}$,\\
    $y_{0}\longrightarrow \text{det}(\delta)y_{0}$,\\
    $A\longrightarrow \text{det}(\alpha)\delta A\delta^{t}$,\\
    $B\longrightarrow \text{det}(\delta)\alpha B\alpha^{t}$.

\end{center}

\noindent For the above element $g.x=(-x_{0},-y^{\prime}_{0},0,B^{\prime})\in X$, $B^{\prime}$ is symmetric and therefore congruent to a diagonal matrix. Hence, for an appropriate choice of $\alpha$ we can get $g_{1}\in G$ with $\beta = \gamma = 0$ such that $g_{1}.(g.x)=(-x^{\prime\prime}_{0},-y^{\prime\prime}_{0},0,\text{diag}(y_{1},y_{2},y_{3}))$ for some scalars $x^{\prime\prime}_{0},y^{\prime\prime}_{0},y_{1},y_{2},y_{3}$. So, any element $x\in X$ can be reduced by the action of $G$ to an element of the form \[y=-x_{0}e_{1}e_{2}e_{3}-y_{0}e_{4}e_{5}e_{6}+y_{1}e_{1}e_{5}e_{6}+y_{2}e_{4}e_{2}e_{6}+y_{3}e_{4}e_{5}e_{3},\]

\noindent and we can see that the entire transformation process is rational over $k(x)$.

\end{proof}

\begin{remark}\label{R3.5}
  If $y=(-x_{0},-y_{0},0,\text{diag}(y_{1},y_{2},y_{3}))\in X$ with $x_{0}\neq 0$ and we take the element $g=\text{diag}(x^{-1}_{0},1,1,x_{0},1,1)\in Sp_{6}$, then $g.y=(-1,-x_{0}y_{0},0,\text{diag}(x^{-1}_{0}y_{1},x_{0}y_{2},x_{0}y_{3}))$. So, in particular, any $x\in X$ can be reduced to an element of the form $(-1,y_{0},0,\text{diag}(y_{1},y_{2},y_{3}))\in X$ over $k(x)$.
\end{remark}

It is very easy to see that $$X^{\prime}=\text{span}\{e_{1}e_{2}e_{3},e_{4}e_{5}e_{6},e_{1}e_{5}e_{6},e_{4}e_{2}e_{6},e_{4}e_{5}e_{3}\}\subset w(V_{2}\otimes \wedge^{2}V_{4})$$ and $J_{1}$ at $y=(-x_{0},-y_{0},0,\text{diag}(y_{1},y_{2},y_{3}))\in X^{\prime}$, has the value $J_{1}(y)=x_{0}^{2}y_{0}^{2}-4x_{0}y_{1}y_{2}y_{3}$. As each $G$-orbit in $X$ intersects $X^{\prime}$, the restriction of $J$ to $X^{\prime}$ is enough to specify the quartic $G$-invariant on $X$. So, we can define $J$ taking $J=J_{1}$ on $X^{\prime}$. In order to get the same invariant as described in (\cite{IJ}, Section $5$), we define $J=-(1/4)J_{1}$ on $X^{\prime}$, which is again a quartic $G$-invariant. If we use the identification of $X$ with $Z$, we can get the quartic invariant $J$ on $X$ as \[J(x)=x_{0}N(B)+y_{0}N(A)+T(A^{\#}, B^{\#})-(1/4)(x_{0}y_{0}-T(A,B))^{2},\]

\noindent where $x=(-x_{0},-y_{0},A,B)\in X$; $N$, $T$ and $\#$ are the cubic norm (the usual determinant in our case), the bilinear trace form and the Freudenthal adjoint defined on the Freudenthal algebra $\mathcal{H}_{3}(k)$, respectively. The same formula holds over fields with characteristic $3$ as $T(A,B)=\text{trace}(AB)$ and $A^{\#}=\text{det}(A)A^{-1}$ (i.e., the transpose of the cofactor matrix of $A$), where $A,B$ are symmetric matrices of order $3$ and $AB$ is the usual product of the matrices $A$ and $B$.

\section{A parametrization of composition algebras}\label{4}

In this section, we classify all the $G$-orbits (i.e., $Sp_{6}$-orbits) in $X-\{0\}$ with the help of (\cite{IJ}) and then prove that each of the $G_{k}$-orbits in $X_{k}-\{0\}$ represents a unique isomorphism class of composition algebras, giving us a description of all such algebras defined over the field $k$ in terms of the orbit space $(X_{k}-\{0\})/G_{k}$. We will also try to see if we can relate our findings in this section to the family of Freudenthal algebras, later in Section \ref{6}. We start with the following result due to Igusa, from (\cite{IJ}, Proposition $7$). We recall that a group $G$ satisfies the Witt condition for a $G$-orbit $U$ means that $G_{k}$ acts transitively on $U_{k}$.

\begin{proposition}[Igusa, 1970]\label{P4.1}

\textit{The $G$-orbits in $X-\{0\}$ are $J^{-1}\{i\}$ with $i\neq 0$ having $-e_{1}e_{2}e_{3}-2(-i)^{1/2}e_{4}e_{5}e_{6}$ as representative and $J^{-1}\{0\}$ decomposes into three orbits represented by $-e_{1}e_{2}e_{3}$, $e_{1}e_{4}e_{3}+e_{5}e_{2}e_{3}$ and $e_{1}e_{4}e_{3}+e_{5}e_{2}e_{3}+e_{1}e_{2}e_{6}$, respectively. The Witt condition is always satisfied only for the orbit represented by $-e_{1}e_{2}e_{3}$, over any field $k$. The stabilizer subgroups are respectively isomorphic to $SL_{3}$ over $k(\sqrt{-i})$, $SL_{3}\ltimes (G_{a})^{6}$ over $k$, a subgroup $H$ of $G$ with two components $H_{0}, H_{1}$ ($H_{0}$ is isomorphic over $k$ to a semidirect product of $SL_{2}\times SL_{2}$ and a connected unipotent group of dimension $5$ and $H_{1}$ has a rational point over $k$) and $SO_{3}\ltimes(G_{a})^{5}$ over $k$ (where the quadratic form of $SO_{3}$ is of index $1$ over $k$).}
    
\end{proposition}

\vspace{1 mm}

\noindent \textbf{Smoothness of the stabilizers:} Let us discuss the smoothness of the stabilizer subgroups of the orbit representatives in the above proposition, which we will need later.

\begin{enumerate}
    \item The stabilizer subgroup of the orbit representative $-e_{1}e_{2}e_{3}-2(\sqrt{-i})e_{4}e_{5}e_{6}\in J^{-1}\{i\}$ (for $i\in k^{\times}$) is isomorphic to $SL_{3}$ as an algebraic group over $k(\sqrt{-i})$, which is a smooth algebraic group (see \cite{KMRT}, Chapter VI, Example $21.11$). We will see later in Theorem \ref{O} that the stabilizer of any element $x\in J^{-1}\{i\}(k)$ is isomorphic to the special unitary group $SU(h)$ as an algebraic group over $k$, for some $k(\sqrt{-i})/k$-hermitian form $h$ of rank $3$. In this case also they are smooth algebraic groups over $k$ (see \cite{KMRT}, Chapter VI, Proposition $21.9$, $21.10$).

    \vspace{1 mm}

    \item The stabilizer of the orbit representative $-e_{1}e_{2}e_{3}$ is isomorphic to $SL_{3}\ltimes (G_{a})^{6}$ as an algebraic group over $k$, and this is also smooth over $k$ as both $SL_{3}$ and $(G_{a})^{6}$ are smooth (see \cite{KMRT}, Chapter VI, Section $21$).

    \vspace{1 mm}

    \item The identity component of the stabilizer subgroup of the element $e_{1}e_{4}e_{3}+e_{5}e_{2}e_{3}$ is isomorphic to a semi-direct product of $SL_{2}\times SL_{2}$ and a connected unipotent group of dimension $5$, as an algebraic group over the base field $k$. Now $SL_{2}\times SL_{2}$ is smooth (see \cite{KMRT}, Chapter VI, Proposition $21.9$, $21.10$), and any connected unipotent group of dimension $5$ is smooth over fields with characteristic zero (see \cite{KMRT}, Chapter VI, Proposition $21.10$). Hence, the stabilizer is smooth if the underlying field $k$ has characteristic zero.

    \vspace{1 mm}

    \item The stabilizer of the element $e_{1}e_{4}e_{3}+e_{5}e_{2}e_{3}+e_{1}e_{2}e_{6}$ is isomorphic to $SO_{3}\ltimes (G_{a})^{5}$ as an algebraic group defined over $k$, where the quadratic form of $SO_{3}$ is of index $1$ over $k$. This group is also smooth as both $SO_{3}$ and $(G_{a})^{5}$ are smooth algebraic groups.
\end{enumerate}

\vspace{2 mm}

The Witt condition is satisfied for all orbits in the above proposition if the underlying field $k$ is algebraically closed. We fix the following notations for the $G$-orbits described above, where $\mathcal{O}(x)$ denotes the orbit of $x\in X$. Let us take

\begin{center}
    $U_{i}=J^{-1}\{i\}$, $i\neq 0$,\\
    $X_{0}=\mathcal{O}(-e_{1}e_{2}e_{3})$,\\
    $X_{1}=\mathcal{O}(e_{1}e_{4}e_{3}+e_{5}e_{2}e_{3}),$\\
    $X_{2}=\mathcal{O}(e_{1}e_{4}e_{3}+e_{5}e_{2}e_{3}+e_{1}e_{2}e_{6})$.
\end{center}

 To avoid notational ambiguity, we will denote the set of $k$-rational points in $X_{r}$ by $X_{r}(k)$ for $r=0,1,2$; and in $U_{i}$ by $U_{i}(k)$ for $i\in k^{\times}$. Then, $X_{0}(k)$ is a single orbit under the $G_{k}$-action as the Witt condition is satisfied by Proposition \ref{P4.1}.

\begin{theorem}\label{q}
   \textit{$X_{1}(k)/G_{k}$ is in one to one correspondence with the set of all isomorphism classes of $2$-dimensional composition algebras defined over $k$.}
\end{theorem}

\begin{proof}
   
    Any element $x\in X_{1}(k)$ is $G_{k}$-equivalent to an element of the form $-e_{1}e_{2}e_{3}+ae_{1}e_{5}e_{6}$, for some $a\in k^{\times}$; and we can also check that two elements $-e_{1}e_{2}e_{3}+ae_{1}e_{5}e_{6}$ and $-e_{1}e_{2}e_{3}+be_{1}e_{5}e_{6}$, for $a,b\in k^{\times}$, are $G_{k}$-equivalent if and only if $ab^{-1}\in k^{\times 2}$ (see \cite{IJ}, page $1025$).

\vspace{2 mm}
    
    We define a map $f:X_{1}(k)/G_{k}\rightarrow k^{\times}/k^{\times 2}$ by $f(\mathcal{O}(x))=a.k^{\times 2}$, where $x=-e_{1}e_{2}e_{3}+ae_{1}e_{5}e_{6}\in X_{1}(k)$. The map is well-defined and injective. To see this, let $\mathcal{O}(x),\mathcal{O}(y)\in X_{1}(k)/G_{k}$ be such that $f(\mathcal{O}(x))=f(\mathcal{O}(y))=a.k^{\times 2}$. This means $g_{1}.x= x_{1}= -e_{1}e_{2}e_{3}+ab^{2}_{1}e_{1}e_{5}e_{6}$ and $g_{2}.y= y_{1}= -e_{1}e_{2}e_{3}+ab^{2}_{2}e_{1}e_{5}e_{6}$, for some $b_{1},b_{2}\in k^{\times}$ and $g_{1},g_{2}\in G_{k}$. It is very easy to see that $g.x_{1}=y_{1}$, where $g=\text{diag}(b^{-1}_{1}b_{2},1,b_{1}b^{-1}_{2},b_{1}b^{-1}_{2},1,b^{-1}_{1}b_{2})\in G_{k}$. Hence, we have $\mathcal{O}(x)=\mathcal{O}(y)$. The surjectivity of $f$ is clear. So, $f$ is a bijection.
    
    \vspace{2 mm}
    
    As $\text{char}(k)\neq 2$, all quadratic field extensions over $k$ (the division composition algebras of dimension $2$) are given by $k(\sqrt{a})$ up to isomorphism, where $a$ varies over the non-trivial elements in $k^{\times}/k^{\times 2}$ and the trivial element will give us the split quadratic algebra, as $k(\sqrt{a})=k[t]/\langle t^{2}-a \rangle \simeq k\times k$ if $a\in k^{\times 2}$. Hence we have the one to one correspondence, where the quadratic algebra corresponding to the orbit of $x=-e_{1}e_{2}e_{3}+ae_{1}e_{5}e_{6}\in X_{1}(k)$ for $a\in k^{\times}$ is given by $k(\sqrt{a})=k[t]/\langle t^{2}-a\rangle$.
    
\end{proof}

Before proceeding to the next results, we first prove the following lemmas.

\begin{lemma}\label{l1}
   \textit{Let $U$ be a $G$-orbit in $X$ and $H\subset G$ be the stabilizer subgroup of an element in $U_{k}$. Also, assume that $H$ is smooth as an algebraic group. Then $U_{k}/G_{k}\longleftrightarrow H^{1}(k,H)$.}
\end{lemma}

\begin{proof}
The exact sequence of cohomology sets is given by (see Proposition \ref{ESC}),

\begin{center}
    $1\longrightarrow H_{k}\longrightarrow G_{k}\longrightarrow (G/H)_{k}\longrightarrow H^{1}(k,H)\longrightarrow H^{1}(k,G)$.
\end{center}

Here we have $H^{1}(k,G)=H^{1}(k,Sp_{6})$, which is trivial, and we can identify the kernel of the map $H^{1}(k,H)\rightarrow H^{1}(k,G)$ with the quotient space of $(G/H)_{k}$ by the action of $G_{k}$ (see Corollary \ref{ESCC}). So, we get a bijection between $U_{k}/G_{k}$ and $H^{1}(k,H)$.
\end{proof}

\begin{lemma}\label{l2}
  \textit{The set of all isometry classes of rank $3$ trivial discriminant quadratic forms defined over $k$, is in bijection with the set of all isometry classes of $2$-fold Pfister forms defined over $k$}.
\end{lemma}

\begin{proof}
    The $2$-fold Pfister forms over $k$ are given by $\langle 1,-a\rangle\otimes\langle 1,-b\rangle\simeq \langle 1,-a,-b,ab\rangle$, for $a,b\in k^{\times}$. We know that two Pfister forms are isometric if and only if the corresponding pure subforms are isometric (see \cite{LTY}, Chapter X). So, these forms are uniquely determined up to isometry by their rank $3$ pure subforms $\langle -a,-b,ab\rangle $, which have trivial discriminants. Conversely, let the quadratic form $\langle x,y,z\rangle, \text{ for } x,y,z\in k^{\times}$, have a trivial discriminant and $xyz=c^{2}$, for some $c\in k^{\times}$. Then $\langle x,y,z\rangle\simeq  \langle x,y,xy\rangle $, which is the pure subform of the $2$-fold Pfister form $\langle 1,x\rangle \otimes \langle 1,y\rangle$. Hence, every isometry class of the pure subforms of $2$-fold Pfister forms represents a unique isometry class of trivial discriminant rank $3$ quadratic forms over $k$, and vice versa. So we are done.
\end{proof}

The set of all isometry classes of $2$-fold Pfister forms over the field $k$ classifies all quaternion algebras defined over $k$, up to isomorphism. So, if we can get some particular type of $G_{k}$-orbits in $X_{k}-\{0\}$ in bijection with any of these two sets in the above lemma, we can parametrize the quaternion algebras over $k$ in terms of those orbits. We describe such orbits in our next result.

\begin{theorem}\label{Q}
 \textit{$X_{2}(k)/G_{k}$ is in one to one correspondence with the set of all isomorphism classes of quaternion algebras defined over $k$.}
\end{theorem}

\begin{proof}

     Any element $x\in X_{2}(k)$ is $G_{k}$-equivalent to an element of the form $ae_{4}e_{2}e_{3}+be_{1}e_{5}e_{3}+ce_{1}e_{2}e_{6} \text{ for } a,b,c\in k^{\times}$, i.e., $x\sim (-x_{0},-y_{0},A,B)\in X_{k}$ with $x_{0}=0, y_{0}=0, B=0, A=\text{diag}(a,b,c),\text{det}(A)\neq 0$; and if two elements in this reduced form are equivalent by some $g\in G_{k}$, we can choose that $g$ with $\beta =\gamma= 0$ (see \cite{IJ}, page $1026$).

\vspace{1.5 mm}

    We define the action of $GL_{3}(k)$ on the set of all $3\times 3$ symmetric non-singular matrices by $U.(x_{ij}) = \text{det}(U)^{-1}U(x_{ij})U^{t},$ where $ U\in GL_{3}(k)$ and $(x_{ij})$ is a symmetric non-singular matrix of order $3$. This is the formula for the action of an element $g\in Sp_{6}(k)$, given by $\beta = \gamma = 0, \delta = U$; on an element $(-x_{0},-y_{0},A,B)\in X_{2}(k)$ with $x_{0}=y_{0}=0, B=0, A=(x_{ij})$ and $\text{det}(A)\neq 0$ (see Proposition \ref{P3.4}). Let $S$ denote the set of all $GL_{3}(k)$-orbits under the action defined above. 
    
    \vspace{1.5 mm}
    
    We define a map $X_{2}(k)/G_{k}\rightarrow S$ by $\mathcal{O}(x)\mapsto [A]$, the orbit of $A$, where $x=(0,0,A,0)\in X_{2}(k)\subset X_{k}$. This map is clearly a bijection. Now, we will prove that $S$ admits a bijection with the set $T= \{$Isometry classes of rank $3$ trivial discriminant quadratic forms over $k\}$.

\vspace{1.5 mm}

    First we define a map $f:S\rightarrow T$. Let $V_{1}$ be any non-singular symmetric matrix of order $3$, representing a non-degenerate quadratic form over $k$. Now, $\text{det}(\text{det}(U)^{-1}UV_{1}U^{t})=\text{det}(U)^{-1}\text{det}(V_{1})$, for $U\in GL_{3}(k)$. So, if we choose a matrix $U\in GL_{3}(k)$ with $\text{det}(U)=\text{det}(V_{1})$, then $V_{1}$ is equivalent to a symmetric matrix of determinant $1$, say $W_{1}$. Then we define $f([V_{1}])=[W_{1}]_{T}$, where $[W_{1}]_{T}\in T$ denotes the isometry class of the trivial discriminant rank $3$ quadratic form represented by the symmetric matrix $W_{1}$.

    \vspace{1.5 mm}

    The map $f$ is well defined. To see that, suppose $[V_{1}]=[W_{1}]$ in $S$ and $\text{det}(V_{1})=\text{det}(W_{1})=1$. Then,
    \begin{center}
      $W_{1}=\text{det}(U)^{-1}UV_{1}U^{t}\implies\text{det}(U)=1$.
    \end{center}
    
     So, we get $W_{1}=UV_{1}U^{t}$, which implies $ f([V_{1}])=f([W_{1}])$.
    \vspace{2 mm}
    
    The map $f$ is injective. Let $f([V_{1}])=f([W_{1}])$, for $[V_{1}],[W_{1}]\in S$. We may assume that $\text{det}(V_{1})=\text{det}(W_{1})=1$. Then,
    
    \begin{center}
        $[V_{1}]_{T}=[W_{1}]_{T}$\\
        \vspace{1 mm}
        $\implies$ $\exists$ $U\in GL_{3}(k)$ such that $UV_{1}U^{t}=W_{1}$\\
        \vspace{1 mm}
        $\implies$ $\text{det}(U)^{2}=1$\\
        \vspace{1 mm}
        $\implies \text{det}(U)=1,-1.$
    \end{center}

    If $\text{det}(U)=1$, we can write $\text{det}(U)^{-1}UV_{1}U^{t}=W_{1}$.

    \vspace{1.5 mm}
    
    If $\text{det}(U)=-1$, then $\text{det}(U)^{-1}UV_{1}U^{t}=-W_{1}=\text{det}(-I_{3})^{-1}(-I_{3})W_{1}(-I_{3})^{t}$.
    
    \vspace{1.5 mm}
    
    In both of the cases, $[V_{1}]=[W_{1}]$.

    \vspace{1.5 mm}

    The surjectivity is clear. If $[V_{1}]_{T}\in T$ with $\text{det}(V_{1})=1$, then $f([V_{1}])=[V_{1}]_{T}$ for $[V_{1}]\in S$, and every element in $T$ has a representative of the form $[V_{1}]_{T}$ with $\text{det}(V_{1})=1$.

    \vspace{1.5 mm}

    By Lemma \ref{l2}, the set $T$ admits a bijection with the set of all isometry classes of $2$-fold Pfister forms over $k$, which classify all quaternion algebras over $k$, up to isomorphism; and the formula for the norm can also be obtained by Lemma \ref{l2}, using the diagonal matrix with determinant $1$, which appears as a representative of the corresponding orbit.

    \vspace{2 mm}

    \noindent \textbf{An alternative proof:} We get an alternative proof of this theorem using Lemma \ref{l1}. The stabilizer of the point $x=ae_{4}e_{2}e_{3}+be_{1}e_{5}e_{3}+ce_{1}e_{2}e_{6}\in X_{2}(k)$ for $a,b,c\in k^{\times}$ is isomorphic to $(SO_{3}(q)\ltimes (G_{a})^{5})$ over $k$, where $q$ is the quadratic form represented by the matrix diag$(a,b,c)$ (see \cite{IJ}, page $1026$). If we take $g=\text{diag}(a^{-1},b^{-1},c^{-1},a,b,c)\in Sp_{6}(k)$, then we have \[g.x=(a^{2}b^{-1}c^{-1}e_{4}e_{2}e_{3}+a^{-1}b^{2}c^{-1}e_{1}e_{5}e_{3}+a^{-1}b^{-1}c^{2}e_{1}e_{2}e_{6})\] and the quadratic form represented by diag$(a^{2}b^{-1}c^{-1},a^{-1}b^{2}c^{-1},a^{-1}b^{-1}c^{2})$ has trivial discriminant. So, we can assume that the stabilizers of any element $x\in X_{2}(k)$ is isomorphic to $(SO_{3}(q)\ltimes (G_{a})^{5})$, where $q$ has trivial discriminant. Now, $SO_{3}(q)$ and $(G_{a})^{5}$ are both smooth subgroups of $(SO_{3}(q)\ltimes (G_{a})^{5})$ (see \cite{KMRT}, Chapter VI, Section $21$). Using the exact sequence of Galois cohomology for both subgroups and the fact that $(G_{a})^{5}$ is a normal subgroup of the stabilizer above (see Proposition \ref{ESC}; \cite{SJP}, Chapter I, Section $5$, Proposition $38$), we can get that $H^{1}(k,SO_{3}(q))\longleftrightarrow H^{1}(k,SO_{3}(q)\ltimes (G_{a})^{5})$, where $q$ is a rank $3$ trivial discriminant quadratic form over $k$. The proof then follows from Lemma \ref{l1} and $\ref{l2}$ (see \cite{KMRT}, Chapter VII, (29.29)).

\end{proof}

Therefore, we obtain a parametrization of all composition algebras with dimension $2^{r}$ defined over the field $k$, in terms of the orbit spaces $X_{r}(k)/G_{k}$, for $r=0,1,2$; where $J^{-1}\{0\}-\{0\}=\bigsqcup_{r=0}^{2}X_{r}$, the disjoint union of $X_{0},X_{1}$ and $X_{2}$. Each of these orbits represents an isomorphism class of composition algebras uniquely, and the formula for the norm can also be determined from the orbit representatives. Now, we are left with the case of octonion algebras.

\subsection{The octonion algebras}
Let $X_{3}=X-J^{-1}\{0\}$. Here we also denote the set of all $k$-rational points in $X_{3}$ by $X_{3}(k)$. We will show that each $G_{k}$-orbit in $X_{3}(k)$ represents an isomorphism class of octonion algebras defined over $k$, which are the only composition algebras with exceptional automorphism groups. Here, unlike the previous cases, we will see that there may exist more than one orbit representing the same isomorphism class. For this reason, we mention this case separately. However, if we write $X_{3}(k)$ as the disjoint union of the $G_{k}$-invariant subsets $U_{i}(k)=J^{-1}\{i\}(k)$ for $i\in k^{\times}$, then each $G_{k}$-orbit in a fixed $U_{i}(k)$ represents a unique isomorphism class without repetitions. First, we prove this for $U_{i}(k)$, where $-i\in k^{\times}- k^{\times 2}$.

\begin{theorem}\label{O}
    \textit{Each $U_{i}(k)/G_{k}$, for $-i\in k^{\times}-k^{\times 2}$, is in one to one correspondence with the set of all isomorphism classes of octonion algebras containing $k(\sqrt{-i})$ as a composition subalgebra of dimension $2$, under the isomorphisms which fix $k(\sqrt{-i})$ pointwise.}
\end{theorem}

\begin{proof}
    
    Any element $x$ in $U_{i}(k)\subset X_{k}$ can be reduced by the $G_{k}$-action to an element of the form \[y=-e_{1}e_{2}e_{3}-y_{0}e_{4}e_{5}e_{6}+y_{1}e_{1}e_{5}e_{6}+y_{2}e_{4}e_{2}e_{6}+y_{3}e_{4}e_{5}e_{3}\] (see Remark \ref{R3.5}), for some $y_{0},y_{1},y_{2},y_{3}\in k$. Each $G_{k}$-orbit in $U_{i}(k)$ has one such representative at which the value of $J$ is $i$. So, $J(y)=i$ gives us the relation

    \begin{center}
        
  $y_{1}y_{2}y_{3}=(1/4)y_{0}^{2}+i= N((1/2)y_{0}+\sqrt{-i})$, 

    \end{center}
    
  \noindent where $N$ is the $k(\sqrt{-i})/k$-norm. If we use the matrix notation (see Section \ref{3}) we get $\text{det}(B)=N((1/2)y_{0}+\sqrt{-i})$, where $y=(-1,-y_{0},A,B)\in U_{i}(k)$ and we have $A=0, B=\text{diag}(y_{1},y_{2},y_{3})$. From the above relation we get that the $k(\sqrt{-i})/k$-hermitian form represented by the matrix $B$ has trivial discriminant.\
  
  The stabilizer subgroup of the element $-e_{1}e_{2}e_{3}-2(-i)^{1/2}e_{4}e_{5}e_{6}$ is isomorphic to $SL_{3}$ over $k(\sqrt{-i})$ (see Proposition \ref{P4.1}), where the embedding is given by
  
  \begin{center}
      $SL_{3}\simeq \biggl\{ \left(\begin{array}{cc}
W & 0 \\
0 & (W^{t})^{-1}
\end{array}\right): W\in SL_{3} \biggl\}\subset G$.
  \end{center}
  
  \noindent For the above $y\in U_{i}(k)$, one can compute an element $g_{y}\in G(k(\sqrt{-i}))$ such that $g_{y}.(-e_{1}e_{2}e_{3}-2(-i)^{1/2}e_{4}e_{5}e_{6})=y$. For example, we can take $g_{y}$ given by (see Section \ref{3}) 
  
  \[\alpha=\text{diag}\biggl(\frac{y_{0}+2\sqrt{-i}}{4\sqrt{-i}},1,1\biggr)\in M_{3}(k(\sqrt{-i})),\]
    \vspace{1 mm}
    \[\beta=\text{diag}\biggl(\frac{-2y_{1}}{y_{0}+2\sqrt{-i}}, \frac{-y_{2}}{2\sqrt{-i}}, \frac{-y_{3}}{2\sqrt{-i}}\biggr)\in M_{3}(k(\sqrt{-i})),\]
    \vspace{1 mm}
    \[\gamma= \text{diag}\biggl(\frac{-(y_{0}^{2}+4i)}{8(\sqrt{-i})y_{1}}, \frac{-(y_{0}-2\sqrt{-i})}{2y_{2}},\frac{-(y_{0}-2\sqrt{-i})}{2y_{3}}\biggr)\in M_{3}(k(\sqrt{-i})),\]
    \vspace{1 mm}
    \[\delta=\text{diag}\biggl(1,\frac{y_{0}+2\sqrt{-i}}{4\sqrt{-i}},\frac{y_{0}+2\sqrt{-i}}{4\sqrt{-i}}\biggr)\in M_{3}(k(\sqrt{-i})).\]
  
  Then, the stabilizer subgroup of $y\in U_{i}(k)$ is isomorphic to $g_{y}SL_{3}g_{y}^{-1}$ as an algebraic group over $k(\sqrt{-i})$ and the stabilizer over $k$ is given by the relation

\begin{center}
    $\sigma \biggl(g_{y}\left(\begin{array}{cc}
W & 0 \\
0 & (W^{t})^{-1}
\end{array}\right)g_{y}^{-1}\biggl)=g_{y}\left(\begin{array}{cc}
W & 0 \\
0 & (W^{t})^{-1}
\end{array}\right)g_{y}^{-1}$, 
\end{center}

  
  \noindent where $\sigma$ is the standard involution on $k(\sqrt{-i})$ and $W\in SL_{3}$ over $k(\sqrt{-i})$. Using the above description of $g_{y}$ we compute both sides of the above equation as a product of $2\times 2$ matrices of $3\times 3$ blocks and comparing the $3\times 3$ blocks from both sides we get that the stabilizer subgroup of $y$ is isomorphic over $k$ to the special unitary group of the non-degenerate ternary $k(\sqrt{-i})/k$-hermitian form given by the matrix $B=\text{diag}(y_{1},y_{2},y_{3})$, which is a $k$-form of $SL_{3}$ (see \cite{SJP}, Chapter III, Section $1$, page $126$) and a smooth subgroup of $Sp_{6}$ (see \cite{KMRT}, Chapter VI, Proposition $21.9$, $21.10$). The discriminant of this hermitian form is trivial, as we have seen that $\text{det}(B)=y_{1}y_{2}y_{3}=N((1/2)y_{0}+\sqrt{-i})$ is a $k(\sqrt{-i})/k$-norm. Then, by Lemma \ref{l1} we have $$U_{i}(k)/G_{k}\longleftrightarrow H^{1}(k,SU(h)),$$ where $h$ is the non-degenerate ternary $k(\sqrt{-i})/k$-hermitian form with trivial discriminant represented by the matrix $B$. Now, $H^{1}(k,SU(h))$ can be identified with the set of all isometry classes of non-degenerate ternary $k(\sqrt{-i})/k$-hermitian forms with trivial discriminants (see \cite{KMRT}, Example (29.19)). The latter set has one to one correspondence with the set of all isomorphism classes of octonion algebras over k which contain $k(\sqrt{-i})$ as a subalgebra, under the isomorphisms which restrict to the identity on $k(\sqrt{-i})$ (see \cite{TML}, Theorem $2.2$). Hence, we get the bijection, as claimed.

\end{proof}

\noindent \textbf{The square case:} If $-i\in k^{\times 2}$, we can see from the proof of the above theorem that $G_{k}$ acts transitively on $U_{i}(k)$ (as we get $g_{y}\in G_{k}$); and we have $k(\sqrt{-i})=k[t]/\langle t^{2}+i\rangle\simeq k\times k$, which is the split composition algebra of dimension $2$ with norm $N((x,y))=x^{2}+iy^{2},$ for $x,y\in k$. In this case, the single orbit $U_{i}(k)/G_{k}$ corresponds to the isomorphism class of the split octonion algebra $Zorn(k)$, which is the unique octonion algebra containing $k\times k$ as a subalgebra.

\vspace{3 mm}
\begin{remark}\label{remoct}
    If we consider any point in the orbit space $X_{3}(k)/G_{k}$, it is in $U_{i}(k)/G_{k}$ for some $i\in k^{\times}$ and represents an isomorphism class of octonion algebras over $k$ that contain $k(\sqrt{-i})$ as a subalgebra. As every octonion algebra is obtained by repeated applications of Cayley-Dickson doubling from the field $k$, each of them contains a composition subalgebra of dimension $2$. Therefore, each isomorphism class of octonion algebras can be obtained from $U_{i}(k)/G_{k}$, for some $i\in k^{\times}$.
\end{remark}

\vspace{3 mm}

\noindent \textbf{The octonion algebra determined by an orbit:} If $-i\in k^{\times 2}$, $k(\sqrt{-i})\simeq k\times k$ and the single orbit $U_{i}(k)/G_{k}$ corresponds to the split octonion algebra $Zorn(k)$, which we have already described in detail (see Section \ref{2}). Let us assume that $-i\in k^{\times}-k^{\times 2}$, and $x=(-1,-y_{0},0,B)\in U_{i}(k)/G_{k}$ be an orbit representative with $B$ diagonal and $\text{det}(B)\neq 0$. Then, the corresponding octonion algebra is given by $C=k(\sqrt{-i})\oplus k(\sqrt{-i})^{3}$, with the norm as $N_{C}(a,x)=N(a)+h(x,x),$ where $a\in k(\sqrt{-i}), x\in k(\sqrt{-i})^{3}$, $N$ is the $k(\sqrt{-i})/k$-norm and $h$ is the non-degenerate $k(\sqrt{-i})/k$-hermitian form on $k(\sqrt{-i})^{3}$ given by the diagonal matrix $B$, with respect to the standard basis. This description of octonions also works if $-i\in k^{\times 2}$. For example, take $x=-e_{1}e_{2}e_{3}-e_{4}e_{5}e_{6}$. Then $-J(x)=1/4\in k^{\times 2}$ and $g.x=(-1,0,0,\text{diag}(-1/4,-1,-1))$, where $g\in G_{k}$ is given by $\alpha =\beta =\text{diag}(2^{-1},1,1), \gamma = \text{diag}(-1,-2^{-1},-2^{-1})$ and $\delta = \text{diag}(1,2^{-1},2^{-1})$. We can construct the octonion algebra corresponding to the orbit of $x$ using the $k(\sqrt{-i})/k$-hermitian form given by $\text{diag}(-1/4,-1,-1)$ on $k(\sqrt{-i})^{3}$, similarly as we did above. In general, for $-i\in k^{\times 2}$ we have $x=-e_{1}e_{2}e_{3}-2(-i)^{1/2}e_{4}e_{5}e_{6}\in U_{i}(k)/G_{k}$ as a representative of the single orbit. We can find an element $g\in G_{k}$ with diagonal $\alpha,\beta,\gamma$ and $\delta$ such that $g.x=(-1,-y_{0},0,B)\in U_{i}(k)$ for some scalar $y_{0}$ and $B$ is a symmetric matrix of order $3$ with $\text{det}(B)\neq 0$. Then, we use the non-degenerate ternary $k(\sqrt{-i})/k$-hermitian form defined by $B$ to construct the corresponding octonion, which is always split.

\subsection{The parametrization} Finally summarizing the results (Theorem \ref{q}, Theorem \ref{Q} and Theorem \ref{O}) discussed in this section, we get a description of all composition algebras defined over $k$ in terms of the $Sp_{6}(k)$-orbits in $X_{k}$, which we state in the following result.

\begin{theorem}\label{M}
 \textit{Let $X$ be the $14$-dimensional irreducible representation of $G=Sp_{6}$, defined over the field $k$. We can write $X$ as a disjoint union of $G$-invariant subsets $X=\{0\}\sqcup X_{0}\sqcup X_{1} \sqcup X_{2} \sqcup X_{3}$, where $J^{-1}\{0\}-\{0\}=X_{0}\sqcup X_{1}\sqcup X_{2}$, $X_{3}=\bigsqcup_{i\in k^{\times}}U_{i}$, $U_{i}=J^{-1}\{i\}$ for $i\in k^{\times}$ and each $U_{i}$ is also $G$-invariant. Then,
    \begin{enumerate}
        \item $X_{r}(k)/G_{k}$ has one to one correspondence with the set of all isomorphism classes of composition algebras defined over $k$ with dimension $2^{r}$, for $r=0,1,2$;
        \item If $-i\notin k^{\times 2}$, $U_{i}(k)/G_{k}$ has one to one correspondence with the set of all isomorphism classes of octonion algebras defined over $k$ containing $k(\sqrt{-i})$ as a composition subalgebra of dimension $2$, under the isomorphisms which fix $k(\sqrt{-i})$ pointwise. If $-i\in k^{\times 2}$, then $U_{i}(k)/G_{k}$ is a single orbit, i.e., it represents the unique isomorphism class of the octonion algebras containing $k\times k$ as quadratic subalgebra.
    \end{enumerate}
    In particular, each orbit has a representative that gives us the formula for the norm of the corresponding composition algebra.}
\end{theorem}

\begin{proof}
    The proof follows from Theorem \ref{q}, Theorem \ref{Q}, Theorem \ref{O} and the fact that $X_{0}(k)$ is a single $G_{k}$-orbit.
\end{proof}

\begin{remark} We get the following relations between the cardinality of the $k$-rational orbit spaces (in the above result) and the number of composition algebras defined over the base field $k$ in some special cases, from Theorem \ref{M}.

\begin{enumerate}

    \item If $-i\in k^{\times 2}$, $|U_{i}(k)/G_{k}|=1$. This tells us that there exists a unique octonion algebra containing $k\times k$ as a subalgebra, which is the split octonion algebra.

    \item Let $k$ be an algebraically closed field. In that case, we can get from the computations in the proofs of the previous results (see Theorem \ref{q}, \ref{Q}, \ref{O}), $|X_{1}(k)/G_{k}|=|X_{2}(k)/G_{k}|=|U_{i}(k)/G_{k}|=1, \forall i\in k^{\times}$. Hence, $\exists$ a unique composition algebra over $k$, in each of the dimensions $1,2,4$ and $8$.

    \item Let us take $k=\mathbb{R}$. Then, we can check that $|X_{1}(k)/G_{k}|=2$ (see Theorem \ref{q}), and also $|X_{2}(k)/G_{k}|=2$ (see Theorem \ref{Q}), since there are only two orbits with representatives of the form $(0,0,A,0)\in X_{2}(k)\subset X_{k}$; where $A$ is given by $\text{diag}(-1,-1,1)$ and $\text{diag}(1,1,1)$, respectively. So, there are two composition algebras up to isomorphism in each of the dimensions $2$ and $4$. Again, we can check that $|U_{i}(k)/G_{k}|=2$ (see Theorem \ref{O}), for all $i\in k^{\times}-k^{\times 2}$ and $|k^{\times}/k^{\times 2}|=2$. So, there are exactly two octonion algebras over $k$. 

    \item Let $k$ be a finite field. Then we have $|X_{1}(k)/G_{k}|=2$ (see Theorem \ref{q}). So, the number of quadratic algebras over $k$ is two, up to isomorphism. As $\text{char}(k)\neq 2$, any quadratic form with dimension greater than $2$ is isotropic, and so we have $|X_{2}(k)/G_{k}|=1$ from the calculations in the proof of Theorem \ref{Q}. Similarly, we have $|U_{i}(k)/G_{k}|=1$ (see Theorem \ref{O}), $\forall i\in k^{\times}$. Hence, we have a unique quaternion algebra and octonion algebra defined over $k$.

    \item Let $k$ be an algebraic number field and $r$ be the number of real conjugates of $k$ in which $i>0$, for some fixed $i\in k^{\times}$. Then we can check that $|U_{i}(k)/G_{k}|=2^{r}$ as there are exactly $2^{r}$ octonion algebras over $k$ containing $k(\sqrt{-i})$ as a subalgebra, up to $k(\sqrt{-i})$-isomorphism. In this case, $|X_{1}(k)/G_{k}|=|X_{2}(k)/G_{k}|=\infty$ as there are infinitely many composition algebras over $k$ in each of the dimensions $2$ and $4$.

\end{enumerate}

\end{remark}

\section{The orbit spaces in prehomogeneous vector spaces}\label{5}

In this section, we discuss orbit space decompositions in some prehomogeneous vector spaces. The orbit space in the semi-stable set of a PV parametrizes interesting algebraic structures in many cases, as we have discussed in the Introduction (see Section \ref{I}).

\vspace{2 mm}

Let $H$ be an algebraic group with a representation $W$ defined over the field $k$. Assume that there exists some $H$-invariant homogeneous polynomial $J_{W}$ on $W$ such that each of the sets $J^{-1}_{W}\{i\}(\overline{k})$ is a single $H(\overline{k})$-orbit for any non-zero scalar $i\in \overline{k}$, the algebraic closure of $k$. In this situation, if we take the action of the group $GL_{1}$ on $W$ by scalar multiplication, then the complement of the hypersurface (in $W(\overline{k})$) defined by the polynomial $J_{W}$ is a single $(H\times GL_{1})(\overline{k})$-orbit. So, $(H\times GL_{1},W)$ becomes a PV with $J_{W}$ as a relative invariant on $W$ with respect to the character $\chi_{W}((h,a))=a^{\text{deg}(J_{W})}$ for $h\in H$ and $a\in GL_{1}$. We will mainly discuss the PV's which are of this form.

\vspace{2 mm}

We start with the case where $H=Sp_{6}$ and $W=X$ is the irreducible $Sp_{6}$-representation of dimension $14$. We give an arithmetic interpretation of the orbit space in the corresponding PV. Then we see that we can also take $H=GSp_{6}$ with the same $W$ to obtain another PV, and compute the orbit space for this new PV. Finally, we do the same for $H=\text{Aut}(C)$ and $W=1_{C}^{\perp}$, where $(C,N_{C})$ is any octonion algebra defined over the base field $k$ and $1_{C}\in C$ is the identity. At the end of this section, we have also described the orbit spaces in some other prehomogeneous vector spaces as examples.

\subsection{The orbit space of $(\text{Sp}_{\text{6}}\times \text{GL}_{\text{1}}, \text{X})$} 

 Let us consider $X$ as a rational representation of the algebraic group $Sp_{6}\times GL_{1}$ over the field $k$, where the action of $Sp_{6}$ on $X$ is the same as described in Section \ref{3} and the $GL_{1}$-action is given by scalar multiplication. This representation is a PV with the $Sp_{6}$-invariant quartic homogeneous polynomial $J$ as relative invariant and  $X^{ss}=\{x\in X: J(x)\neq 0\}=\bigsqcup_{i\in k^{\times}}U_{i}$ (see Theorem \ref{M}) is the set of all semi-stable points (we will see later that it is a Zariski dense open orbit in $X$). We denote the set of all $k$-rational points in $(Sp_{6}\times GL_{1})$ and $X^{ss}$ by $(Sp_{6}\times GL_{1})(k)$ and $X^{ss}_{k}$, respectively. We show that the $k$-rational orbit space of this PV parametrizes composition algebras, as the $G_{k}$-orbits do in Theorem \ref{M}.

\vspace{2 mm}
 
 \noindent \textbf{The orbits in $(X^{ss}_{k})^{c}$:} We have $(X^{ss})^{c}=J^{-1}\{0\}=\{0\}\sqcup X_{0}\sqcup X_{1}\sqcup X_{2}$ (see Theorem \ref{M}). We can check by direct computations with the $G_{k}$-orbit representatives in $X_{0}(k),X_{1}(k)$ and $X_{2}(k)$ (see Theorem \ref{q} and Theorem \ref{Q}), that for any $a\in GL_{1}(k)$ and $x\in J^{-1}\{0\}(k)-\{0\}$, $ax=g.x$ for some $g\in Sp_{6}(k)$. For example, any $k$-rational $G_{k}$-orbit representative in $X_{2}(k)$ can be taken to be of the form $x=y_{1}e_{4}e_{2}e_{3}+y_{2}e_{1}e_{5}e_{3}+y_{3}e_{1}e_{2}e_{6}$, where $y_{1},y_{2},y_{3}\in k^{\times}$. Let $g\in Sp_{6}(k)$ be the element given by $\beta=\gamma =0$ and $\alpha = aI_{3}\text{ for } a\in k^{\times}$ (so automatically $\delta=a^{-1}I_{3}$). Then we get $g.x=ax$. So, the orbit decomposition of $J^{-1}\{0\}(k)-\{0\}$ remains the same (same as in Theorem \ref{M}) under the $(Sp_{6}\times GL_{1})(k)$-action and we have the following result.
 \begin{theorem}\label{5.1}
     \textit{The non-zero orbits in $(X^{ss}_{k})^{c}/(Sp_{6}\times GL_{1})(k)$ are in bijection with the set of all isomorphism classes of composition algebras defined over the field $k$, with dimension less than or equal to $4$.}

 \end{theorem}

 \begin{proof}
     As we have discussed above, the orbit decomposition in $(X^{ss}_{k})^{c}$ under the $(Sp_{6}\times GL_{1})(k)$-action remains the same as under the $Sp_{6}(k)$-action. In particular, $$X_{r}(k)/(Sp_{6}\times GL_{1})(k)\longleftrightarrow X_{r}(k)/Sp_{6}(k),$$ for $r=0,1,2$. Hence, we get the bijection from Theorem \ref{M}.
 \end{proof}

 \vspace{2 mm}

 \noindent \textbf{The orbits in $X^{ss}_{k}$:} In this case, with the $k$-rational orbit representatives in Theorem \ref{O} we can check that the orbit decomposition in each $U_{i}(k)\subset X^{ss}_{k}$ for $i\in k^{\times}$, remains the same (same as in Theorem \ref{M}) as well, under the new action. Let $$x=-e_{1}e_{2}e_{3}-y_{0}e_{4}e_{5}e_{6}+y_{1}e_{1}e_{5}e_{6}+y_{2}e_{4}e_{2}e_{6}+y_{3}e_{4}e_{5}e_{3}\in U_{i}(k)$$ be a representative of a $G_{k}$-orbit in $U_{i}(k)$. Then $J(x)=i\in k^{\times}$ and for any $a\in GL_{1}(k)$ we have $J(ax)=a^{4}i$. We consider the following two cases:
 \vspace{1 mm}
 \begin{enumerate}
     \item \textbf{$a^{4}\neq 1$:} In this case, $ax\in U_{a^{4}i}(k)$. Since $k(\sqrt{-a^{4}i})=k(\sqrt{-i})$ and so $U_{a^{4}i}(k)/G_{k}\longleftrightarrow U_{i}(k)/G_{k}$ (see Theorem \ref{O}), the $GL_{1}(k)$-action identifies $U_{a^{4}i}(k)/G_{k}$ with $U_{i}(k)/G_{k}$ without disturbing the orbit decomposition in $U_{i}(k)$ under the $G_{k}$-action.

     \vspace{2 mm}

     \item \textbf{$a^{4}=1$:} In this case, $ax\in U_{i}(k)$ and $U_{i}(k)$ is a single orbit if $-i\in k^{\times 2}$. So, we may assume that $-i\in k^{\times}-k^{\times 2}$. If $a=1$, $ax=x\in U_{i}(k)$. If $a=-1$, then $ax=g.x\in U_{i}(k)$, where $g=\text{diag}(-1,-1,-1,-1,-1,-1)\in Sp_{6}(k)$. Let $a\in k^{\times}$ be such that $a^{2}=-1$. Then we can check that the non-degenerate ternary $k(\sqrt{-i})/k$-hermitian forms corresponding to $x$ and $ax$ (see the proof of Theorem \ref{O}), both are isometric to the hermitian form given by the matrix $\text{diag}(y_{1},y_{2},y_{3})$. So, $x$ and $ax$ must be $G_{k}$-equivalent, and $\exists $ $g\in Sp_{6}(k)$ such that $ax=g.x\in U_{i}(k)$.

 \end{enumerate}

\vspace{1 mm}

 Hence, the $(Sp_{6}\times GL_{1})(k)$-action on $X^{ss}_{k}$ identifies $U_{a^{4}i}(k)$ with $U_{i}(k)$ for $i,a\in k^{\times}$, without disturbing the orbit decomposition in $U_{i}(k)$ under the $Sp_{6}(k)$-action. So, we have the following result.

 \begin{theorem}\label{5.2}
      \textit{$X^{ss}_{k}/(Sp_{6}\times GL_{1})(k)=\bigsqcup_{i\in k^{\times}/k^{\times 4}}U_{i}(k)/G_{k}$. So, each point in this orbit space represents an isomorphism class of octonion algebras over $k$, and all such algebras can be obtained from these orbits.}
      
 \end{theorem}

\begin{proof}
  From the above discussion, we find that the $(Sp_{6}\times GL_{1})(k)$-action identifies $U_{ia^{4}}(k)$ with $U_{i}(k)$ and the orbit decomposition in $U_{i}(k)$ remains the same as in the $Sp_{6}(k)$-action, for $i,a\in k^{\times}$. So, we get $$X^{ss}_{k}/(Sp_{6}\times GL_{1})(k)=\bigsqcup_{i\in k^{\times}/k^{\times 4}}U_{i}(k)/G_{k}.$$ Now, each $U_{i}(k)/G_{k}$ parametrizes all octonion algebras containing $k(\sqrt{-i})$ as a subalgebra (see Theorem \ref{M}). Hence, the result follows from Remark \ref{remoct}.
\end{proof}

\vspace{2 mm}

\noindent \textbf{$(\text{Sp}_{\text{6}}\times \text{GL}_{\text{1}}, \text{X})$ is a PV:} The above calculations for the orbits in $X^{ss}_{k}$ show that $X^{ss}$ is a single orbit over $\overline{k}$ (the algebraic closure of $k$), as $U_{i}$'s are single $G$-orbits if $-i$ is a square and $ij^{-1}\in \overline{k}^{\times 4},\forall i,j\in \overline{k}^{\times}$. So, this representation is an irreducible PV with $J$ as a relative invariant.


\begin{remark} 
The relative invariant $J$ defines a projective variety $X_{J}$ in the projective space $\mathbb{P}(X)$. Now, for $ x\in X_{k}, a\in k^{\times}$ we have $ax\in \mathcal{O}(x)$, the $(Sp_{6}\times GL_{1})(k)$-orbit of $x$. So, the above results give us an interpretation for the $k$-rational orbit spaces in the projective variety $X_{J}$ and its complement $X^{c}_{J}\subset \mathbb{P}(X)$, under the $(Sp_{6}\times GL_{1})$-action. As a result, we get an arithmetic interpretation for the $k$-rational orbit space in the projective space $\mathbb{P}(X)$.

\end{remark}

\subsection{The orbit space of $(\text{GSp}_{\text{6}}\times \text{GL}_{\text{1}}, \text{X})$}

Let $GSp_{6}\simeq GL_{1}\ltimes Sp_{6}$ be the group of all similitudes of the standard non-degenerate alternating bilinear form defined on the off-diagonal subspace $V\subset Zorn(k)$, over the field $k$. Then $(GSp_{6}\times GL_{1},X)$ is also a PV over $k$ with the same relative invariant $J$, where the action of $GL_{1}$ is given by scalar multiplication and $GSp_{6}$ has its induced action on $X\subset \wedge^{3}V$. We have $X^{ss}=\{x\in X:J(x)\neq 0\}=\bigsqcup_{i\in k^{\times}}U_{i}$ as the semi-stable set, and we will see that it is a Zariski dense open orbit in $X$. We can check $GSp_{6}\simeq H_{1}\ltimes Sp_{6}$, where \[H_{1}=\biggl\{h_{a}=\left(\begin{array}{cc}
aI_{3} & 0 \\
0 & I_{3}
\end{array}\right)\in GSp_{6}:a\in GL_{1}\biggl\}\simeq GL_{1}.\] We denote the set of all $k$-rational points in $(GSp_{6}\times GL_{1})$ and $X^{ss}$ by $(GSp_{6}\times GL_{1})(k)$ and $X^{ss}_{k}$, respectively. Now, we proceed to the orbit space decomposition of this representation.

\vspace{2 mm}

\noindent \textbf{The orbits in $(X^{ss}_{k})^{c}$:} In this case, we also have $(X^{ss})^{c}=J^{-1}\{0\}=\{0\}\sqcup X_{0}\sqcup X_{1}\sqcup X_{2}$ (see Theorem \ref{M}). We can check by direct computations that for $x\in J^{-1}\{0\}(k)-\{0\}$ and $h_{a}\in H_{1}(k) \text{ with } a\in GL_{1}(k)$, $\exists$ $g\in Sp_{6}(k)$ such that $h_{a}.x=g.x$. For example, let us take $x=y_{1}e_{4}e_{2}e_{3}+y_{2}e_{1}e_{5}e_{3}+y_{3}e_{1}e_{2}e_{6}$ as a $G_{k}$-orbit representative in $X_{2}(k)$, where $y_{1},y_{2},y_{3}\in k^{\times}$. Then, $h_{a}.x=a^{2}x=g.x$, where $g=\text{diag}(a^{2},a^{2},a^{2},a^{-2},a^{-2},a^{-2})\in Sp_{6}(k)=G_{k}$. We can also check that for any $a\in GL_{1}(k)$ and $x\in J^{-1}\{0\}(k)-\{0\}$, $\exists$ $h\in G_{k}$ such that $ax=h.x$; similarly as we did for the PV $(Sp_{6}\times GL_{1},X)$. Hence, the orbit decomposition of $J^{-1}\{0\}(k)-\{0\}$ remains the same (same as in Theorem \ref{M}) and we have the following result.

\begin{theorem}\label{5.3}
    The non-zero orbits in $(X^{ss}_{k})^{c}/(GSp_{6}\times GL_{1})(k)$ are in bijection with the set of all isomorphism classes of composition algebras defined over the field $k$, with dimension less than $8$.
\end{theorem}

\begin{proof}
    From the discussion above, we find that the orbit space in $(X^{ss}_{k})^{c}$ under the $(GSp_{6}\times GL_{1})(k)$-action remains the same as in the $Sp_{6}(k)$-action. So, in this case also $$X_{r}(k)/(GSp_{6}\times GL_{1})(k)\longleftrightarrow X_{r}(k)/Sp_{6}(k),$$ for $r=0,1,2$ and we have the bijection, as claimed (see Theorem \ref{M}). 
\end{proof}

\vspace{2 mm}

\noindent \textbf{The orbits in $X^{ss}_{k}$:} Let $x\in U_{i}(k)\subset X^{ss}_{k}$ for some $i\in k^{\times}$, i.e., $J(x)=i$. Then we can easily see that $J(h_{a}.x)=a^{6}i$, for $h_{a}\in H_{1}(k),a\in k^{\times}$. If we take $g=(h_{a},a^{-1})\in (GSp_{6}\times GL_{1})(k)$, then we get $J(g.x)=a^{2}J(x)$ and proceeding similarly as we did for $(Sp_{6}\times GL_{1},X)$, we can check that the orbit decomposition in each $U_{i}(k)$ for $ i \in k^{\times }$, remains the same (same as in Theorem \ref{M}), only $U_{i}(k)$ gets identified with $U_{i a^{2}}(k)$ for $i,a\in k^{\times}$. So, we get $$X^{ss}_{k}/(GSp_{6}\times GL_{1})(k)=\bigsqcup_{i\in k^{\times}/k^{\times 2}}U_{i}(k)/G_{k}.$$ Let $Ex(2)$ denote the set of all isomorphism classes of quadratic algebras defined over $k$, i.e., $Ex(2)\longleftrightarrow X_{1}(k)/(GSp_{6}\times GL_{1})(k)$. We define a map $$\gamma_{X}:X^{ss}_{k}/(GSp_{6}\times GL_{1})(k)=\bigsqcup_{i\in k^{\times}/k^{\times 2}}U_{i}(k)/G_{k}\rightarrow Ex(2),$$ by $\gamma_{X}(\mathcal{O}(x))=k(\sqrt{-i})=k[t]/\langle t^{2}+i\rangle $, for $\mathcal{O}(x)\in U_{i}(k)/G_{k}$. Then $\gamma_{X}$ is surjective and we have the following result.

\begin{theorem}\label{5.4}
    Let $\gamma_{X}:X^{ss}_{k}/(GSp_{6}\times GL_{1})(k)\rightarrow Ex(2)$ be the map as defined above. This map has the following properties:

    \begin{enumerate}
        \item $\gamma_{X}^{-1}(k\times k)=(GSp_{6}\times GL_{1})(k).(e_{1}e_{2}e_{3}+e_{4}e_{5}e_{6})$.
        \item If $K/k$ is a quadratic extension over $k$, $\gamma_{X}^{-1}(K)$ is in bijection with the set of all isomorphism classes of octonion algebras containing $K$ as a subalgebra, under isomorphisms which fix $K$ pointwise.
    \end{enumerate}
\end{theorem}

\begin{proof}
    $(1)$ This part follows very easily. We have $J(e_{1}e_{2}e_{3}+e_{4}e_{5}e_{6})=-1/4$ and $-(-1/4)\in k^{\times 2}$. So, $\mathcal{O}(e_{1}e_{2}e_{3}+e_{4}e_{5}e_{6})\mapsto k\times k$. Conversely, if $\mathcal{O}(x)\mapsto k\times k$ for some $x\in X^{ss}_{k}$, then $-J(x)\in k^{\times 2}$ and $J(x).k^{\times 2}=(-1/4).k^{\times 2}\in k^{\times}/k^{\times 2}$. Again, $U_{i}(k)/G_{k}$ is a singleton set if $-i\in k^{\times 2}$. Hence, by the above calculations for the orbits in $X^{ss}_{k}$ we have $x\in \mathcal{O}(e_{1}e_{2}e_{3}+e_{4}e_{5}e_{6})$. Therefore, $\gamma^{-1}_{X}(k\times k)=(GSp_{6}\times GL_{1})(k).(e_{1}e_{2}e_{3}+e_{4}e_{5}e_{6})$.
    
    \vspace{1 mm}

    $(2)$ In this case, we have $K\simeq k(\sqrt{-i})$, for some $-i\in k^{\times}-k^{\times 2}$ (as char$(k)\neq 2$). The proof follows from the orbit space decomposition $X^{ss}_{k}/(GSp_{6}\times GL_{1})(k)=\bigsqcup_{i\in k^{\times}/k^{\times 2}}U_{i}(k)/G_{k}$ and Theorem \ref{O}.
\end{proof}

It is very easy to see that the above theorem is an alternative statement of the main theorem in (\cite{YA1}; Theorem 1.10, Theorem 4.11), due to A. Yukie. The identity components of the stabilizers of the orbit representatives can also be computed to be the same, as mentioned by Yukie in his result. Moreover, Yukie proved the result by using some special type of elements from the group $GSp_{6}\times GL_{1}$ which does not seem possible to choose from the group $Sp_{6}\times GL_{1}$ as he mentioned in (\cite{YA1}; $(4.1)$, page $44$). So, he avoided using the fact that the underlying vector space $X$ is a representation of the groups $Sp_{6}\times GL_{1}$ or $Sp_{6}$ (see \cite{YA1}; $(4.1)$, page $44$, and Section $1$, page $35$-$36$), and used Galois cohomology with the stabilizer subgroups (in $(GSp_{6}\times GL_{1})$) of the points in $X^{ss}_{k}$ (see \cite{YA1}, proof of Theorem $4.11$). However, we proved this result using the $k$-rational orbit space decomposition of $X$ as a representation of the group $Sp_{6}$ and identifying the $Sp_{6}(k)$-orbits in $X^{ss}_{k}$ under the action of the new elements $a\in k^{\times}, h_{a}\in H_{1}(k)$ which generate $(GSp_{6}\times GL_{1})(k)$ with the subgroup $Sp_{6}(k)$ (this can be seen from the above computations). 

\vspace{1 mm}

So, we get an alternate proof of Yukie's result, and this gives us an arithmetic interpretation of the $(GSp_{6}\times GL_{1})(k)$-orbits in the semi-stable set $X^{ss}_{k}$ in terms of the octonion algebras defined over $k$.

\vspace{2 mm}

\noindent \textbf{$(\text{GSp}_{\text{6}}\times \text{GL}_{\text{1}}, \text{X})$ is a PV:} Similar to the case of $(Sp_{6}\times GL_{1},X)$, from the above computations it follows that $X^{ss}$ is a single orbit over $\overline{k}$ and we have $J$ as the relative invariant. Hence, this representation is an irreducible PV.

\vspace{3 mm}

 We have a canonical map $$\gamma_{ss} : X^{ss}_{k}/(GSp_{6}\times GL_{1})(k)\rightarrow \{\text{Isomorphism classes of octonion algebras over } k \},$$ as each orbit represents a unique octonion algebra over $k$. Clearly, $\gamma_{ss} $ is surjective and if we take the isomorphism class $[C]_{k}$ of an octonion algebra $(C,N_{C})$ over $k$, then $\gamma_{ss}^{-1}\{[C]_{k}\}=\{\mathcal{O}((-1,-y_{0},0,\text{diag}(y_{1},y_{2},y_{3})))\in U_{i}(k)/G_{k}:i\in k^{\times}/k^{\times 2}, k(\sqrt{-i})\subset C $ and $N_{C}\simeq N\perp h $, where N is the norm on $k(\sqrt{-i})$ and $h$ represents the trace form of the non-degenerate ternary $k(\sqrt{-i})/k$-hermitian form defined by the matrix $\text{diag}(y_{1},y_{2},y_{3})\}$. From this one can easily see the following result.

\begin{theorem}\label{fiber}
    $\gamma_{ss}^{-1}\{[C]_{k}\}\longleftrightarrow \{$Isomorphism classes of quadratic subalgebras of the octonion algebra $C\}$.
\end{theorem}

\begin{proof}
    The proof follows from the above description of $\gamma_{ss}^{-1}\{[C]_{k}\}$ and the orbit space decomposition $X^{ss}_{k}/(GSp_{6}\times GL_{1})(k)=\bigsqcup_{i\in k^{\times}/k^{\times 2}}U_{i}(k)/G_{k}$.
\end{proof}

\vspace{1 mm}

Let $C_{1}, C_{2}$ be two composition algebras over $k$ with $F_{1}\subset C_{1}$, $F_{2}\subset C_{2}$ as composition subalgebras, respectively. We write them as the $2$-tuples $(C_{1},F_{1})$ and $(C_{2},F_{2})$. An \textit{isomorphism} between $(C_{1},F_{1})$ and $(C_{2},F_{2})$ means that there is an isomorphism $f:C_{1} \to C_{2} $ such that $f(F_{1})=F_{2}$.

\begin{corollary}\label{tuple}
    The orbit space $X^{ss}_{k}/(GSp_{6}\times GL_{1})(k)$ is in bijection with the $2$-tuples of composition algebras $(C,F)$ up to isomorphism, where $C$ is an octonion algebra over $k$ with $F\subset C$ as a quadratic subalgebra.
\end{corollary}

\begin{proof}
    The proof follows from Theorem \ref{fiber}.
\end{proof}
    

\vspace{2 mm}

 \noindent \textbf{A cohomological interpretation:} Now, we discuss a cohomological interpretation of the orbit space in the above PV, as an example that follows from Proposition \ref{ESC} and Corollary \ref{ESCC}. As we have seen in Section \ref{4}, the stabilizer subgroups of the orbit representatives in Proposition \ref{P4.1} are all smooth algebraic groups over a field $k$ with $\text{char}(k)=0$.

 \vspace{1 mm}
 
 Let $C$ be an octonion algebra over $k$ and $F\subset C$ a subalgebra. By $\text{Aut}(C/F)$ we denote the group of algebra automorphisms of $C$ fixing $F$ pointwise and $\text{Aut}(C,F)$ be the group of algebra automorphisms mapping $F$ to $F$. If $k(\sqrt{-i})\subset C$ for some $i\in k^{\times}$, it is easy to check that $$\text{Aut}(C,k(\sqrt{-i}))\simeq \text{Aut}(C/k(\sqrt{-i}))\rtimes \mathbb{Z}_{2}\simeq SU(h)\rtimes \mathbb{Z}_{2}$$ (see \cite{JN}, Theorem $3$), where $h$ is a non-degenerate trivial discriminant $k(\sqrt{-i})/k$-hermitian form of rank $3$ defined on $k(\sqrt{-i})^{\perp}\subset C$. The cohomology set $H^{1}(k, \text{Aut}(C,k(\sqrt{-i})))$ classifies the tuples $(C_{1},F_{1})$ up to isomorphism, where $C_{1}$ is an octonion and $F_{1}\subset C_{1}$ is a composition subalgebra of dimension $2$ (see \cite{KMRT}, Chapter VII, Subsection 29.C, ($29.12$)). From the above discussion and previous observations we get the following result over any field $k$ with characteristic zero.
 

\begin{theorem}\label{5.6}
    Let $k$ be a field with $\text{char}(k)=0$. We have $X-\{0\}=X_{0}\sqcup X_{1}\sqcup X_{2}\sqcup X^{ss}$. Then
    
    \begin{enumerate}
        \item $X_{r}(k)/(GSp_{6}\times GL_{1})(k)\longleftrightarrow H^{1}(k, H_{r})$, where $H_{r}\subset Sp_{6}$ are stabilizers of some point in $X_{r}(k)$, for $r=0,1,2$;
        \item $X^{ss}_{k}/(GSp_{6}\times GL_{1})(k)\longleftrightarrow H^{1}(k, \text{Aut}(C,F))$, where $C$ is an octonion algebra and $F\subset C$ is a composition subalgebra of dimension $2$.
        
    \end{enumerate}
\end{theorem}

\begin{proof}\
        $(1)$ The proof of the first part follows from Proposition \ref{ESC} and Corollary \ref{ESCC}, since the stabilizers $H_{r}\subset Sp_{6}$ (for $r=0,1,2$) in Proposition \ref{P4.1} are all smooth subgroups of $Sp_{6}$ (see Section \ref{4}; \cite{KMRT}, Chapter VI, Section $21$) and we have the following bijection (see Theorem \ref{5.3}) \[X_{r}(k)/(GSp_{6}\times GL_{1})(k)\longleftrightarrow X_{r}(k)/Sp_{6}(k).\]

        \vspace{1 mm}

        $(2)$ As we have discussed above, the set $H^{1}(k, \text{Aut}(C,F))$ can be identified with the set of all isomorphism classes of the pairs $\{(C,F):C$ is an octonion and $F\subset C $ is a composition subalgebra of dimension $2\}$. Hence, the result follows from Corollary \ref{tuple}.

\end{proof}

\begin{remark}
    In the above result, $(1)$ holds over any arbitrary field $k$ (char$(k)\neq 2$) except for $r=1$, and $(2)$ also holds over any field with characteristic different from $2$ (see Section \ref{4}).
\end{remark}

So, the above result gives us an interpretation for the orbit spaces in the semi-stable set $X^{ss}_{k}$ and the hypersurface $(X^{ss}_{k})^{c}$, in terms of cohomology sets.


\begin{remark} 
We have some observations about the irreducible prehomogeneous vector space $(GSp_{6}\times GL_{1},X)$, which we state in the following remarks.

\begin{enumerate}

    \item \textbf{The stabilizers of the generic points:} We can easily verify that the stabilizer of a generic point $x=(-1,-y_{0},0,B)\in X^{ss}_{k}\subset X$, where $B=\text{diag}(y_{1},y_{2},y_{3})$ and $\text{det}(B)\neq 0$, contains the subgroup of $GSp_{6}\times GL_{1}$ given by $\{(aI_{6},a^{-3})\in GSp_{6}\times GL_{1}: a\in GL_{1}\}\simeq GL_{1}$. So, $SU(h)\times GL_{1}\subset \text{Stab}(x)$, where $h$ is the rank $3$ trivial discriminant $k(\sqrt{-i})/k$-hermitian form represented by $B$ if $J(x)=i$ (see Theorem \ref{O}). Now, $\text{dim}(GSp_{6}\times GL_{1})-\text{dim}(\text{Stab}(x))=\text{dim}(X)=14$ implies $\text{dim}(\text{Stab}(x))=9=\text{dim}(SU(h)\times GL_{1})$. Hence, we find that the identity component of $\text{Stab}(x)$ is isomorphic to $SU(h)\times GL_{1}$. In (\cite{YA1}, Proposition $4.5$), Yukie proved that the identity component of $\text{Stab}(x)$ has index $2$ in $\text{Stab}(x)$. So, we get $\text{Stab}(x)\simeq (SU(h)\times GL_{1})\rtimes \mathbb{Z}_{2}$.

    \item \textbf{A parametrization of trivial discriminant rank 3 hermitian forms:} For $i\in k^{\times}/k^{\times 2}$ the orbit space $U_{i}(k)/G_{k}$ admits a bijection with the isometry classes of trivial discriminant $k(\sqrt{-i})/k$-hermitian forms of rank $3$ (see Theorem \ref{O}). So, the orbit space in the semi-stable set classifies all trivial discriminant $K/k$-hermitian forms of rank $3$, where $K$ is any quadratic algebra over $k$.

    \item In general, if $C$ is any composition algebra and $D\subset C$ is a subalgebra, $H^{1}(k,\text{Aut}(C,D))$ classifies all the $k$-forms of the tuple $(C,D)$ up to isomorphism, as defined earlier. So, these cohomology sets can be identified with the sets consisting of $2$-tuples of orbits from $X_{k}/(GSp_{6}\times GL_{1})(k)$, by appropriate choices. For example, if we take C to be an octonion algebra and $D\subset C$ a quaternion subalgebra, then $\text{Aut}(C,D)\simeq SL_{1}(D)\rtimes \text{Aut}(D)$ and $H^{1}(k,\text{Aut}(C,D))\hookrightarrow (X^{ss}_{k}/(GSp_{6}\times GL_{1})(k)) \times (X_{2}(k)/(GSp_{6}\times GL_{1})(k))$. Similarly, if we take $C$ to be a quaternion algebra and $D\subset C$ a quadratic subalgebra, then $H^{1}(k,\text{Aut}(C,D))\hookrightarrow (X_{2}(k)/(GSp_{6}\times GL_{1})(k)) \times (X_{1}(k)/(GSp_{6}\times GL_{1})(k))$.

    \item Similar to the previous subsection, we can also get an interpretation for the $k$-rational orbit spaces in the projective space $\mathbb{P}(X)$, the projective variety $X_{J}\subset \mathbb{P}(X)$ (i.e., set of all zeros of $J$ in $\mathbb{P}(X)$) and its complement $X^{c}_{J}\subset \mathbb{P}(X)$, under the $(GSp_{6}\times GL_{1})$-action.
\end{enumerate}

\end{remark}

\subsection{The PV associated to the octonion algebras}\label{SS5.3} Let $(C,N_{C})$ be any octonion algebra defined over the field $k$ and $G_{C}=\text{Aut}(C)$ be the group of all algebra automorphisms, which is a simple connected group of type $G_{2}$. Let $1_{C}\in C$ be the identity and $V_{C}=1^{\perp}_{C}\subset C$ be its orthogonal complement with respect to the norm $N_{C}$. Then $V_{C}$ is the space of all trace zero elements in $C$ and an irreducible representation of $G_{C}$. To avoid notational ambiguity, we denote the set of all $k$-rational points in $G_{C}$ and $V_{C}$ by $G_{C}(k)$ and $V_{C}(k)$, respectively. If we take $J_{C}$ to be the restriction of the octonion norm $N_{C}$ to $V_{C}$, it is $G_{C}$-invariant. Let $V^{ss}_{C}=\{v\in V_{C}:J_{C}(v)\neq 0\}$, which is the same as $V_{C}-\{0\}$ for division algebras. We denote the set of $k$-rational points in $V^{ss}_{C}$ by $V^{ss}_{C}(k)$. Then we have the following result.

\begin{theorem}\label{5.7}
    The $G_{C}(k)$-orbits in $V_{C}(k)-\{0\}$ are the fibers $J_{C}^{-1}\{i\},i\in k$.
\end{theorem}

\begin{proof}
    Let $a,b\in V_{C}(k)-\{0\}$ be such that $J_{C}(a)=J_{C}(b)=i\neq 0$. Since $a$ is orthogonal to $1_{C}$, the restriction of the norm $N_{C}$ to the space $K_{1}=k.1_{C}\oplus k.a$ is non-singular. So $K_{1}$ is a composition subalgebra of $C$ with dimension $2$ (see \cite{SV}, Chapter $1$, Proposition $1.2.3$), and the norm on $K_{1}$ is given by $\langle 1,i\rangle$. Similarly, we can check that $K_{2}=k.1_{C}\oplus k.b$ is also a composition subalgebra of $C$ with norm $\langle 1,i\rangle$ and $K_{1}\simeq K_{2}$ by the isomorphism determined by sending $a$ to $b$. This isomorphism can be extended to a $k$-isomorphism $\phi:C\rightarrow C$ (see \cite{SV}, Chapter $1$, Corollary $1.7.3$) such that $\phi(a)=b$ and $\phi \in G_{C}(k)$. So, $a$ and $b$ are in the same $G_{C}(k)$-orbit.
    
    \vspace{2 mm}
    
    If $C$ is a division algebra, $J^{-1}_{C}\{0\}(k)=\{0\}$. If $C$ is split, $J^{-1}_{C}\{0\}(k)$ contains non-zero elements. Let $a,b\in V_{C}(k)-\{0\}$ be two such elements with $J_{C}(a)=J_{C}(b)=0$. In this case, we cannot proceed with the quadratic algebras as above. Then there exists two quaternion algebras $Q_{1},Q_{2}\subset C$ containing $a$ and $b$, respectively (see \cite{SV}, Chapter $1$, Proposition $1.6.4$). Both quaternions are isomorphic to the matrix algebra $(M_{2}(k),\text{det})$, where `$\text{det}$' denotes the determinant and $a\in Q_{1},b\in Q_{2}$ correspond to the nilpotent elements, as $a^{2}=b^{2}=0$. So, there exists an isomorphism $\phi_{1}:Q_{1}\rightarrow Q_{2}$. Then, $\phi_{1}(a)$ is nilpotent in $Q_{2}\simeq M_{2}(k)$ and there exists an isomorphism $\phi_{2}:Q_{2}\rightarrow Q_{2}$ such that $\phi_{2}(\phi_{1}(a))=b$; since $Q_{2}\simeq M_{2}(k)$ and all nilpotent elements in $M_{2}(k)$ are equivalent to the matrix $$\left(\begin{array}{cc}
0 & 1 \\
0 & 0
\end{array}\right)$$ (i.e., the unique Jordan canonical form), under the action of the automorphism group. Now, $\phi_{2}\phi_{1}$ extends to an automorphism of $C$ (see \cite{SV}, Chapter $1$, Corollary $1.7.3$), which maps $a$ to $b$. So, we get that all non-zero elements in $J^{-1}_{C}\{0\}(k)$ form a single orbit.

\end{proof}


If we look into (\cite{HJ}, Chapter I, Proposition $1.2$), it is very easy to see that Haris has considered the same representation over number fields with $C$ as the split octonion (so $\text{Aut}(C)$ also splits). From the above result, we get an alternate proof of Haris's result and its generalization for the non-split cases defined over an arbitrary field with characteristic different from $2$. 

\begin{remark}
    For quaternion algebras, we can also get the orbit space decomposition in the orthogonal complement of the identity under the action of the automorphism group, similarly as above. Even one can observe that the above theorem has been proved using the Skolem-Noether theorem for the quaternion algebras (central simple algebras of degree $2$) and the result (\cite{SV}, Chapter $1$, Corollary $1.7.3$).
\end{remark}

The following result gives us the orbit decomposition in the entire octonion algebra $(C,N_{C})$ over $k$ under the action of $\text{Aut}(C)$.

\begin{theorem}
    If $C$ is any octonion algebra defined over $k$, the orbits in $C$ under the action of its automorphism group are given by
    
    \begin{enumerate}
        \item $\lambda.1_{C}$ for $\lambda\in k$; 
        
        \item $U_{ij}=\{x\in C:N_{C}(x)=i; b_{N_{C}}(1_{C},x)=j\}$ for $i,j\in k$.
    \end{enumerate}
\end{theorem}

\begin{proof}
    Clearly, $\lambda.1_{C}$ is a single orbit for each $\lambda\in k^{\times}$, and for $x\in V_{C}$ the result follows from Theorem \ref{5.7}. So we consider $x\in C$ which is linearly independent of $1_{C}$ and $x\notin V_{C}$, i.e., $b_{N_{C}}(1_{C},x)\neq 0$, where $b_{N_{C}}$ is the polar form of the norm $N_{C}$. Any element $y\in C$ with $b_{N_{C}}(1_{C},y)=b_{N_{C}}(1_{C},x)$ and $N_{C}(x)=N_{C}(y)$ is equivalent to $x$. To see this, we consider the subspaces \[K_{1}=k.1_{C}\oplus k.x\subset C \text{ and } K_{2}=k.1_{C}\oplus k.y\subset C.\] If $K_{1}$ is non-singular with respect to the norm $N_{C}$, $K_{1}$ is a composition subalgebra of $C$ and so is $K_{2}$ (see \cite{SV}, Proposition $1.2.3$). Then $K_{1}$ is isomorphic to $K_{2}$ as a composition algebra by the isomorphism determined by the map $x\mapsto y$ (see \cite{SV}, Chapter $1$, Section $3$, $(1.9)$), and it can be extended to an automorphism of $C$ that maps $x$ to $y$ (see \cite{SV}, Corollary $1.7.3$). If $K_{1}$ is singular with respect to the norm $N_{C}$, then we have \[K_{1}=k.1_{C}\oplus k.(x-\frac{1}{2}b_{N_{C}}(1_{C},x)1_{C})\text{ and } N_{C}(x-\frac{1}{2}b_{N_{C}}(1_{C},x)1_{C})=0.\] The element $x-\frac{1}{2}b_{N_{C}}(1_{C},x)1_{C}$ is a non-zero element in $C$ with zero norm and trace, and must be contained within a split quaternion subalgebra of $C$ (see \cite{SV}, Proposition $1.6.4$) as a nilpotent element, i.e., $C$ is also split. This subalgebra contains the element $x$. We also get the same for $y$ using the element $y-\frac{1}{2}b_{N_{C}}(1_{C},y)1_{C}$. From the computations in the proof of Theorem \ref{5.7}, we get an automorphism of $C$ that maps $x-\frac{1}{2}b_{N_{C}}(1_{C},x)1_{C}$ to $y-\frac{1}{2}b_{N_{C}}(1_{C},y)1_{C}$ and this isomorphism obviously maps $x$ to $y$. So, we get the following result.
\end{proof}


\vspace{3 mm}

\noindent \textbf{$(\text{G}_{\text{C}}\times \text{GL}_{\text{1}}, \text{V}_{\text{C}})$ is a PV:} We consider the representation $(G_{C}\times GL_{1},V_{C})$, where the $GL_{1}$-action is given by scalar multiplication. Then, from the above computations, it is clear that this representation is an irreducible PV with $J_{C}$ as relative invariant and $V^{ss}_{C}$ as the semi-stable set, which is a Zariski open dense orbit over $\overline{k}$. We denote the set of all $k$-rational points in $(G_{C}\times GL_{1})$ by $(G_{C}\times GL_{1})(k)$. As $J_{C}$ has degree $2$, we get the following result.

\begin{theorem}\label{T5.10}
    If $C$ is the split octonion algebra, $V^{ss}_{C}(k)/(G_{C}\times GL_{1})(k)\longleftrightarrow k^{\times}/k^{\times 2}$. So, the orbit space in the semi-stable set is in bijection with the set of all isomorphism classes of composition algebras with dimension $2$ over the field $k$.
\end{theorem}

\begin{proof}
    The norm of the split octonion algebra is hyperbolic. Then $J_{C}$ also contains hyperbolic subforms, i.e., $J_{C}$ attains every element in $k$ on $V_{C}(k)$. So, every element in $k^{\times}$ is attained by $J_{C}$ for some non-zero element in $V^{ss}_{C}(k)$. Now the proof follows from Theorem \ref{5.7} and the fact that $J_{C}$ has degree $2$.
\end{proof}

\begin{remark}
    If $C$ is a division octonion algebra, then the orbit space $V^{ss}_{C}(k)/(G_{C}\times GL_{1})(k)$ is in bijection with $\{N_{C}(x):x\in V_{C}(k)-\{0\}\}/k^{\times 2}\subset k^{\times}/k^{\times 2}$ (see Theorem \ref{5.7}).
\end{remark}

\noindent \textbf{The stabilizers of the generic points:} Let $a\in V^{ss}_{C}(k)$ and $H=\text{Stab}_{(G_{C}\times GL_{1})}(a)$, for any octonion algebra $C$. As $H\subset G_{C}\times GL_{1}$, let $(g,\lambda)\in H$ for $g\in G_{C}$ and $\lambda\in GL_{1}$. Then $J_{C}((g,\lambda).a)=\lambda^{2}J_{C}(a)$ and we get $\lambda^{2}=1$, as $(g,\lambda)$ fixes $a$. This implies $\lambda=1,-1$, and so we must have $g.a=a,-a$, respectively. If $\lambda=1$, $g$ belongs to the subgroup of $G_{C}$ that fixes the quadratic subalgebra $K_{1}=k.1_{C}\oplus k.a\subset C$ pointwise. If $\lambda=-1$, we must have $g.a=-a$ and such $g\in G_{C}$ exist, as $a\mapsto -a$ defines an algebra automorphism on $K_{1}$ (see \cite{SV}, Chapter $1$, Corollary $1.7.3$). Then the restriction of $g$ to $K_{1}$ is an algebra automorphism of $K_{1}$. So, we find that the stabilizer subgroup of $a\in V^{ss}_{C}(k)$ is isomorphic to the group of all algebra automorphisms of $C$ that maps $K_{1}$ onto itself. Then we have $H\simeq \text{Aut}(C,K_{1})\simeq SU(h)\rtimes \mathbb{Z}_{2} $ (for details, see \cite{JN}, Theorem $3$), where $h$ is a non-degenerate $K_{1}/k$-hermitian form of rank $3$ with trivial discriminant and $\text{dim}(G_{C}\times GL_{1})-\text{dim}(H)=15-8=7=\text{dim}(V_{C})$. In particular, $C\simeq K_{1}\oplus K_{1}^{\perp}$ and $h$ is a hermitian form on $K_{1}^{\perp}$ such that $N_{C}(x,y)=N_{K_{1}}(x)+h(y,y),x\in K_{1}, y\in K_{1}^{\perp}$ and $N_{K_{1}}$ denotes the $K_{1}/k$-norm.

\vspace{3 mm}

\noindent\textbf{Cohomological interpretation:} If we consider the smooth stabilizer subgroup $H\subset G_{C}\times GL_{1}$ (see \cite{KMRT}, Chapter VI, Proposition $21.9$, $21.10$) of the generic point $a\in V^{ss}_{C}(k)$ (as described above), then the exact sequence of cohomology sets is given by (see Proposition \ref{ESC}) $$1\rightarrow H_{k}\rightarrow (G_{C}\times GL_{1})(k)\rightarrow ((G_{C}\times GL_{1})/H)(k)\rightarrow H^{1}(k,H)\rightarrow H^{1}(k,G_{C}\times GL_{1}),$$ and we get $V^{ss}_{C}(k)/(G_{C}\times GL_{1})(k)\longleftrightarrow \text{ker}\{H^{1}(k,H)\rightarrow H^{1}(k,G_{C}\times GL_{1})\}$ (see Corollary \ref{ESCC}).

\begin{remark}\label{CE}
Let us consider an octonion division algebra $(C,N_{C})$ defined over the field $k$. If $V_{C}^{\ast}$ is the dual of $V_{C}$, $(G_{C}\times GL_{6}, V_{C}^{\ast}\otimes V_{6})$ is a PV (with the natural actions) castling equivalent to $(G_{C}\times GL_{1},V_{C}\otimes V_{1})\simeq(G_{C}\times GL_{1},V_{C})$, as $G_{C}\subset SO(J_{C})\subset GL(V_{C})$ (see \cite{SV}, Chapter $2$) and $1+6=7$; where $V_{6}$ and $V_{1}$ denote the vector spaces of dimensions $6$ and $1$, respectively. As we can see in (\cite{IJ1}, Proposition $3.1$), we can identify $V_{C}^{\ast}\otimes V_{6}$ with the space of all $7\times 6$ matrices $M_{7\times 6}$ as a representation of $G_{C}\times GL_{6}$ and the semi-stable set in this PV corresponds to the matrices of maximal ranks. Since there are non-zero matrices with rank less than the maximal rank, there are at least two $k$-rational orbits in the complement of the semi-stable set in this PV. But for $(G_{C}\times GL_{1},V_{C})$ there is no $k$-rational orbit in the complement of the semi-stable set other than $\{0\}$. So, this shows that the $k$-rational orbit space in the complement of the semi-stable set may not be preserved under the castling equivalence.
\end{remark}


\subsection{More prehomogeneous vector spaces} 

In the rest of this section, we describe the orbit space decompositions and arithmetic interpretations of the orbit spaces in the semi-stable sets for some PV's, which follows easily from (\cite{HJ}) and (\cite{IJ1}). Similar to these two references, we assume $\text{char}(k)=0$ in this subsection. We refer the reader to (\cite{IJ1}) and (\cite{SK}) for more details.

\vspace{2 mm}

\noindent \textbf{A parametrization of the octonions:} Let us consider the group $G_{1}=GL_{7}$ defined over $k$ and the induced action of $G_{1}$ on $Y_{1}=\wedge^{3}V_{7}$, where $V_{7}$ is the vector space of dimension $7$ with $\{e_{r}\}^{7}_{r=1}$ as the standard basis. Let $\rho : G_{1}\rightarrow GL(Y_{1})$ denote this representation. We can check that $(G_{1},Y_{1})$ is a prehomogeneous vector space with a relative invariant $J_{2}$ of degree $7$ (see \cite{IJ1}, Section $6$, type ($6$); \cite{SK}, page $83$-$86$) and the semi-stable set $Y^{ss}_{1}=\{ x\in Y_{1}: J_{2}(x)\neq 0\}$ is a Zariski open dense $G_{1}$-orbit. Let $\Tilde{G_{1}}=\rho(G_{1})$ and $H\subset \Tilde{G_{1}}$ be the stabilizer of some fixed point $y_{1}\in Y^{ss}_{1}(k)$. Then $H$ is isomorphic to a connected simple group of type $G_{2}$ over $k$ and we get the following result from (\cite{IJ1}, Section $6$, type(6)), which can be seen using Proposition \ref{ESC} and Corollary \ref{ESCC}.

\begin{theorem}
    The $G_{1}(k)$-orbits in $Y^{ss}_{1}(k)$ are in bijection with the set of all isomorphism classes of octonion algebras defined over the field $k$.
\end{theorem}

Now we have the following identification of $(G,X)$, i.e., $(Sp_{6},X)$ inside $(G_{1},Y_{1})$:
 \begin{center}
     $G\simeq \biggl\{ \left(\begin{array}{cc}
g & 0 \\
0 & 1
\end{array}\right) \in GL_{7}: g\in Sp_{6}\biggl\}\subset G_{1}$
\end{center}

 \noindent and $X\hookrightarrow \wedge^{3}V_{6}\hookrightarrow \wedge^{3}V_{7}=Y_{1}$, where $V_{6}$ is the vector space of dimension $6$ and we identify it inside $V_{7}=\text{span}\{e_{r}\}^{7}_{r=1}$ as $V_{6}=\text{span}\{e_{r}\}^{6}_{r=1}\subset V_{7}$, which gives us the required embedding. Then the representation $\rho |_{G}:G\rightarrow GL(Y_{1})$ has $X$ as one of its irreducible components, which gives us all composition algebras in Theorem \ref{M}.

 \vspace{3 mm}
    
    \noindent \textbf{A parametrization of the quaternions:} We can check that $(SL_{5}\times GL_{3},Y_{2})$ is a PV (see \cite{IJ1}, Section $6$, type($10$); \cite{SK}, page $94$-$97$ for details), where $Y_{2}=(\wedge^{2}V_{5})^{3}$ and $V_{5}$ is the $5$-dimensional vector space. The relative invariant has degree $15$. As we can see in (\cite{IJ1}, type $(10)$), $$Y_{2}^{ss}(k)/(SL_{5}\times GL_{3})(k)\longleftrightarrow H^{1}(k, \text{Aut}(SL_{2})).$$ So, the orbits in the semi-stable set are in bijection with the set of all isomorphism classes of quaternion algebras defined over $k$ (see \cite{SJP}, Chapter III, Section $1$, page 125).

    \vspace{3 mm}
    
    \noindent \textbf{k-forms of $\text{SL}_{3}$:} $(GL_{8},Y_{4})$ is also a PV (see \cite{IJ1}, Section $6$; \cite{SK}, page $87$-$90$), where $Y_{4}=\wedge^{3}V_{8}$, $V_{8}$ is the vector space of dimension $8$ and $GL_{8}$ has the induced action on $Y_{4}$. The relative invariant has degree $16$. We can see from (\cite{IJ1}, Section $6$, type $(7)$), $$Y_{4}^{ss}(k)/GL_{8}(k)\longleftrightarrow H^{1}(k,\text{Aut}(SL_{3})),$$ which classifies the $k$-forms of $SL_{3}$ (see \cite{SJP}, Chapter III, Section $1$, page $125$ for details).
    
    \vspace{3 mm}
    
    \noindent \textbf{($\text{E}_{\text{7}}\times \text{GL}_{1},\text{V}_{\text{56}}$):} We can get the orbit space decomposition for the PV $(E_{7}\times GL_{1}, V_{56})$ (see \cite{SK}, page $140$; \cite{IJ1}, type ($29$)) from the computations in (\cite{HJ}, Chapter II, Proposition $2.1$, $2.2$); where $E_{7}$ is the simple connected $k$-split group of type $E_{7}$, $V_{56}$ is the irreducible representation of $E_{7}$ with dimension $56$ and the $GL_{1}$-action is given by scalar multiplication. Here, the relative invariant has degree $4$. The $k$-rational orbits in the semi-stable set $V^{ss}_{56}(k)$ under the $E_{7}(k)$-action are the fibers of the non-zero scalars in $k^{\times}$, under the relative invariant (see \cite{HJ}, Chapter II, Proposition $2.1$). So, we have $$V^{ss}_{56}(k)/(E_{7}\times GL_{1})(k)\longleftrightarrow k^{\times}/k^{\times 4},$$ which is the cohomology set $H^{1}(k,\mu_{4})$, where $\mu_{4}$ denotes the group of $4$-th roots of unity (see \cite{KMRT}, Chapter VII, Section $30$, ($30.1$)). The orbit space decomposition of the complement of the semi-stable set can also be determined easily (see \cite{HJ}, Chapter II, Proposition $2.2$). For details, we refer the reader to (\cite{HJ}) and (\cite{SK}).

\section{Some observations about Freudenthal algebras}\label{6}

In this section, we explain what we get for the Freudenthal algebras from the descriptions of the composition algebras discussed in the previous two sections (Section \ref{4} and Section \ref{5}). We also discuss some PV's where the underlying vector spaces of the representations are the Freudenthal algebras, and we compute the orbit spaces in detail for the Freudenthal division algebras of dimension $9$. If we have a look at these PV's, it is very easy to see that we obtain them from the same construction as described in the beginning of section \ref{5} by taking $W$ as a Freudenthal algebra, $J_{W}$ to be the cubic norm defined on $W$ and $H$ as the group of isometries of the cubic norm. For this section, we always assume that the characteristic of the underlying field $k$ is different from $2$ and $3$.

\vspace{2 mm}

Let $(C,N_{C})$ be any composition algebra defined over the field $k$ and $\mathcal{H}_{3}(C,\Gamma)$ be the reduced Freudenthal algebra for some $\Gamma=\text{diag}(\gamma_{1},\gamma_{2},\gamma_{3})\in GL_{3}(k)$. If we fix $\Gamma$, then $\mathcal{H}_{3}(C_{1},\Gamma)\simeq \mathcal{H}_{3}(C_{2},\Gamma)$ iff $C_{1}\simeq C_{2}$ (see \cite{KMRT}, Chapter IX for details). So, in this case, all algebras of this type can be parametrized by Theorem \ref{M} and the results discussed in Section \ref{5} as well. Now, we investigate the case where $C$ is fixed and $\Gamma$ varies.

\vspace{2 mm}

We have the bilinear trace form $T$ defined on $\mathcal{H}_{3}(C,\Gamma)$. The isomorphism class of $\mathcal{H}_{3}(C,\Gamma)$ is determined by the isometry class of $T$ (see \cite{KMRT}, Chapter IX). We can check that \[T\simeq \langle 1,1,1\rangle \perp b_{N_{C}}\otimes \langle \gamma_{3}^{-1}\gamma_{2},\gamma_{1}^{-1}\gamma_{3},\gamma_{2}^{-1}\gamma_{1} \rangle \simeq \langle 1,1,1\rangle \perp b_{N_{C}}\otimes \langle -a, -b, ab \rangle , \] for some $a,b\in k^{\times}$ and $b_{N_{C}}$ is the bilinear form associated to the norm $N_{C}$ of the composition algebra $C$. Let $q=b_{N_{C}}\otimes \langle\langle a,b\rangle\rangle$. Then the isometry class of $T$ uniquely determines the isometry classes of $N_{C}$ and $q$. Conversely, the isometry classes of $N_{C}$ and $q$ uniquely determine the isometry class of $T$ (see \cite{KMRT}, Chapter IX). Hence, we get the following results for $C=k$, the $1$-dimensional composition algebra; and $C=k(\sqrt{-i})$, the $2$-dimensional composition algebras, $i\in k^{\times}$.

\begin{theorem}
    \textit{The isomorphism classes of the reduced Freudenthal algebras of the form $\mathcal{H}_{3}(k,\Gamma)$ are in one to one correspondence with the isomorphism classes of quaternion algebras defined over $k$, i.e., $X_{2}(k)/G_{k}$ in Theorem \ref{Q}.}
\end{theorem}

\begin{proof}
    As $k$ is the only composition algebra with dimension $1$, the bilinear trace forms for the algebras $\mathcal{H}_{3}(k,\Gamma)$ are uniquely determined by the $2$-fold Pfister forms $\langle\langle a,b\rangle\rangle$, as discussed above. Hence, the result follows.
\end{proof}

So, the above result gives us a parametrization of all Freudenthal algebras with dimension $6$, since any such algebra in dimension $6$ is always reduced over the field $k$ (see \cite{GPR}, Proposition $39.17$, $46.1$ and Theorem $46.8$). However, this result is a weaker form of (Proposition $55.6$, \cite{GPR}).

\begin{theorem}\label{6.2}
     \textit{For any $i\in k^{\times}$, the isomorphism classes of the reduced Freudenthal algebras of the form $\mathcal{H}_{3}(k(\sqrt{-i}),\Gamma)$ are in one to one correspondence with the set of all isomorphism classes of octonion algebras containing $k(\sqrt{-i})$ as a subalgebra, under the isomorphisms which fix $k(\sqrt{-i})$ pointwise, i.e., $U_{i}(k)/G_{k}$ in Theorem \ref{O}.}
\end{theorem}

\begin{proof}
From the earlier discussion, the bilinear trace forms for algebras of this type are uniquely determined by the isometry classes of the $3$-fold Pfister forms that contain $\langle 1, i\rangle$ as a subform (see \cite{EL}, Theorem $2.7$), i.e., isomorphism classes of octonion algebras that contain $k(\sqrt{-i})$ as a subalgebra. Hence, we have the bijection, as claimed.   
\end{proof}

From the above result, we can get a parametrization of all reduced Freudenthal algebras with dimension $9$, which we state in the following result.

\begin{theorem}\label{6.3}
    The set of all isomorphism classes of reduced Freudenthal algebras with dimension $9$ over the field $k$, is in bijection with the orbit space $X^{ss}_{k}/(GSp_{6}\times GL_{1})(k)$ in Theorem \ref{5.6}.
\end{theorem}

\begin{proof}
    The proof follows from Theorem \ref{6.2} and the fact that $X^{ss}_{k}/(GSp_{6}\times GL_{1})(k)=\bigsqcup_{i\in k^{\times}/k^{\times 2}}U_{i}(k)/G_{k}$, as we have seen in Section \ref{5}.
\end{proof}

For the higher dimensional cases, we cannot proceed with our parametrizations of composition algebras. It will be interesting to find a parametrization which gives us all Freudenthal algebras defined over $k$.

\begin{remark} 
Theorem \ref{6.2} and Theorem \ref{6.3} give us another interpretation of the orbit spaces $X^{ss}_{k}/(Sp_{6}\times GL_{1})(k)$ and $X^{ss}_{k}/(GSp_{6}\times GL_{1})(k)$, in terms of the reduced Freudenthal algebras of dimension $9$.
\end{remark}

\vspace{1 mm}

Now we consider the reduced Freudenthal algebras $\mathcal{H}_{3}(C,\Gamma)$ where $C$ is a quaternion or octonion algebra and $\Gamma=\text{diag}(\gamma_{1},\gamma_{2},\gamma_{3})\in GL_{3}(k)$. We assign an element $(\mathcal{O}(x),\mathcal{O}(y))$ from the set $X_{2}(k)/(GSp_{6}\times GL_{1})(k)\times X_{2}(k)/(GSp_{6}\times GL_{1})(k)$ to the $2$-tuple of Pfister forms $(N_{x},N_{x}\otimes N_{y})$ which determines a unique algebra of the form $\mathcal{H}_{3}(C,\Gamma)$, where $C$ is a quaternion algebra and $N_{x},N_{y}$ denote the norms of the quaternion algebras associated to the orbits $\mathcal{O}(x),\mathcal{O}(y)$, respectively (see Theorem \ref{5.3}). All reduced Freudenthal algebras of this form can be obtained from these tuples. But the assignment is not injective. Similarly, all algebras of the form $\mathcal{H}_{3}(C,\Gamma)$, where $C$ is an octonion, can be obtained from the set $X^{ss}_{k}/(GSp_{6}\times GL_{1})(k) \times X_{2}(k)/(GSp_{6}\times GL_{1})(k)$ by the same process, and this assignment is also not injective. The same can be done for $\text{dim}(C)=1,2$.

\vspace{2 mm}

\noindent \textbf{Tits constructions for reduced Freudenthal algebras of dimension 27:} Any Freudenthal algebra of dimension $27$ can be obtained by Tits constructions of first type or second type (for details see \cite{KMRT}, Chapter IX). In the previous section, we defined a map $\gamma_{ss} : X^{ss}_{k}/(GSp_{6}\times GL_{1})(k)\rightarrow \{$Isomorphism classes of octonion algebras over $k \} $, which is surjective. We have just described above how any reduced Freudenthal algebra of dimension $27$ can be obtained from the set $X^{ss}_{k}/(GSp_{6}\times GL_{1})(k) \times X_{2}(k)/(GSp_{6}\times GL_{1})(k)$. We can easily see that $\gamma_{ss}^{-1}(Zorn(k))\times X_{2}(k)/(GSp_{6}\times GL_{1})(k)$ produces all the first Tits constructions for reduced Freudenthal algebras of dimension $27$ and its complement in $ X^{ss}_{k}/(GSp_{6}\times GL_{1})(k)\times X_{2}(k)/(GSp_{6}\times GL_{1})(k)$ produces reduced Freudenthal algebras of dimension $27$ which can be obtained by second Tits constructions only (see \cite{KMRT}, Chapter IX, Proposition $(40.5)$). For more details on Tits constructions of first and second type, we refer to (\cite{KMRT}, Chapter IX).


\subsection{PV associated to Freudenthal algebras} Let us consider the Freudenthal algebra $\mathcal{A}=\mathcal{H}_{3}(Zorn(k),I_{3})$ with dimension $27$, defined over the field $k$ and $E_{6}$ be the group of isometries of the cubic norm defined on it, which is a $k$-split group of type $E_{6}$. Then $(E_{6}\times GL_{1},\mathcal{A})$ is a PV with the identity element as a generic point, and the action of $GL_{1}$ is given by scalar multiplication. The relative invariant is the cubic norm itself. So, the semi-stable set $\mathcal{A}^{ss}$ is the set of all elements in $\mathcal{A}$ with non-zero norms. It is easy to check that the identity component of the stabilizer subgroup of the identity element $1\in \mathcal{A}$ is $\text{Aut}(\mathcal{A})$ (see \cite{JN1}, Theorem $4$), which is a simple $k$-split group of type $F_{4}$. So, $$\text{dim}(E_{6}\times GL_{1})-\text{dim}(\text{Stab}_{(E_{6}\times GL_{1})}(1))=79-52=\text{dim}(\mathcal{A}).$$ Under the $E_{6}(k)$-action there are two non-zero orbits (i.e., orbits other than $\{0\}$) in the hypersurface defined by the relative invariant, with representatives $\text{diag}(1,0,0)$ and $\text{diag}(1,1,0)$ (see \cite{KS}, Section $2$, Proposition $12$); containing the elements with Jordan rank $1$ and $2$, respectively. The $E_{6}(k)$-orbits in the semi-stable set are the fibers of the non-zero scalars under the cubic norm, with representatives $\text{diag}(1,1,i), i\in k^{\times}$; which consist of the elements with Jordan rank $3$ (see \cite{KS}, Section $2$, Proposition $12$). So, we can get the $k$-rational orbit space in this PV. If we consider any split composition algebra $(C,N_{C})$, the corresponding reduced Freudenthal algebra $\mathcal{H}_{3}(C,I_{3})$, and $NP(\mathcal{H}_{3}(C,I_{3}))$ as the group of isometries of the cubic norm defined on $\mathcal{H}_{3}(C,I_{3})$; then $$(NP(\mathcal{H}_{3}(C,I_{3}))\times GL_{1},\mathcal{H}_{3}(C,I_{3}))$$ is a PV with the $GL_{1}$-action given by scalar multiplication. The cubic norm is the relative invariant, and the identity element is a generic point. The orbit decompositions for these PV's are the same as in the earlier case (see \cite{KS} for details). If every element in the field $k$ is a square, we also get the same for the PV $(NP(\mathcal{H}_{3}(k,I_{3}))\times GL_{1}, \mathcal{H}_{3}(k,I_{3}))$. We refer to (\cite{KS}), (\cite{JN1}) and (\cite{JN2}) for more details on the actions of the groups $NP(\mathcal{H}_{3}(C,I_{3}))$ on $\mathcal{H}_{3}(C,I_{3})$.

\vspace{1 mm}

If we take $C$ to be any division composition algebra in the above discussion, then we can easily see that $(NP(\mathcal{H}_{3}(C,I_{3}))\times GL_{1},\mathcal{H}_{3}(C,I_{3}))$ are PV's as well, with the same actions and relative invariants described above. We can get the orbit decompositions for these PV's from (\cite{JN1}, Theorem $8$, Theorem $9$), for $\text{dim}(C)=1,2,4$ and (\cite{JN2}, Proposition $1$, Theorem $12$), for $\text{dim}(C)=8$. We may also look at (\cite{RRS}) when the group $NP(\mathcal{H}_{3}(C,I_{3}))$ is a group of type $E_{6}$.

\vspace{1 mm}

In general, if $\mathcal{J}$ is any Freudenthal algebra over the field $k$ with $N_{\mathcal{J}}$ as the cubic norm defined on it, then $(NP(\mathcal{J})\times GL_{1},\mathcal{J})$ is a prehomogeneous vector space with $N_{\mathcal{J}}$ as relative invariant, where $NP(\mathcal{J})$ is the group of isometries of $N_{\mathcal{J}}$ and the $GL_{1}$-action is given by scalar multiplication.

\vspace{2 mm}

\noindent \textbf{Division algebras of dimension 9:} Now, as an example, we discuss the case of Freudenthal division algebras with dimension $9$. Any algebra of this type is of the form $$\mathcal{J}=\mathcal{H}(B,\sigma)=\{x\in B:\sigma(x)=x\},$$ where $B$ is a central division algebra of degree $3$ over a quadratic extension $K$ of $k$ and $\sigma$ is an involution of the second kind (see \cite{KMRT}, Chapter IX for details). For the algebra $B$, we cannot have an anti-automorphism which fixes the center $K$ elementwise, as the order of $B$ in the Brauer group is not $2$. So, the group of all isometries of the cubic norm $N_{\mathcal{J}}$ defined on $\mathcal{J}$ is generated by the linear transformations of the form $$x\mapsto \lambda \sigma(a)xa, \text{ for }x\in \mathcal{J}, a\in B^{\times}, \lambda\in GL_{1},$$ where $\lambda^{3}N(\sigma(a)a)=1$ and $N$ is the generic norm defined on $B$ (see \cite{JN1}, Theorem $8$). We denote this group by $NP(\mathcal{J})$, as we have done earlier. Then, $(NP(\mathcal{J})\times GL_{1},\mathcal{J})$ is a PV, where the action of $GL_{1}$ is given by scalar multiplication and the cubic norm $N_{\mathcal{J}}$ is a relative invariant. So, the semi-stable set is given by $$\mathcal{J}^{ss}=\{x\in \mathcal{J}: N_{\mathcal{J}}(x)\neq 0 \}.$$ We denote the set of all $k$-rational points in $NP(\mathcal{J})\times GL_{1}$ and $\mathcal{J}^{ss}$ by $(NP(\mathcal{J})\times GL_{1})(k)$ and $\mathcal{J}^{ss}_{k}$, respectively.

\vspace{1 mm}

Let $(B,\sigma)$ be as above. We define the \textit{group of similitudes} in $B$ (see \cite{KMRT}, Chapter III) as $$GU(B,\sigma)=\{x\in B:\sigma(x)x\in k^{\times}\}.$$

\vspace{1 mm}

Now, $B$ has a natural right $B$-module structure. For a non-zero element $x\in \mathcal{J}$ we define a $B$-hermitian form on the right $B$-module $B$ by $$\langle x\rangle_{B}(a_{1},a_{2})=\sigma(a_{1})xa_{2}, \text{ for } a_{1},a_{2}\in B$$ (see \cite{HKRT}, Section $2$). We call them \textit{rank $1$ $B$-hermitian forms} on $B$. We say that two rank $1$ $B$-hermitian forms $\langle x\rangle_{B}$ and $\langle y\rangle_{B}$, for $x,y\in \mathcal{J}$, are \textit{isometric} if there exists some $z\in B^{\times}$ such that $x=\sigma(z)yz$, and \textit{similar} if $x=\lambda\sigma(z)yz$ for some $\lambda \in k^{\times}, z\in B^{\times}$. We refer to (\cite{HKRT}) for more details.

\begin{theorem}
    The orbit space in the semi-stable set $\mathcal{J}^{ss}_{k}/(NP(\mathcal{J})\times GL_{1})(k)$ is in bijection with the similarity classes of rank $1$ $B$-hermitian forms defined on $B$.
\end{theorem}

\begin{proof}
    As $\mathcal{J}$ is a division algebra, all non-zero elements in $\mathcal{J}$ are in the semi-stable set $\mathcal{J}^{ss}$. Now, the group $NP(\mathcal{J})$ is generated by the linear transformations of the form $$x\mapsto \lambda \sigma(a)xa, \text{ for }x\in \mathcal{J}, a\in B^{\times}, \lambda \in GL_{1},$$ where $\lambda^{3}N(\sigma(a)a)=1$ (see \cite{JN1}, Theorem $8$), and the action of $GL_{1}$ on $\mathcal{J}$ is given by scalar multiplication. So, it is enough to consider the action of the above linear transformations on $\mathcal{J}$ removing the restriction $\lambda^{3}N(\sigma(a)a)=1$ (see \cite{JN1}, Theorem $3$, Theorem $8$), which is nothing more than the action of the structure group (group of similarities of the norm $N_{\mathcal{J}}$) of $\mathcal{J}$. Hence, the $k$-rational orbit space in the semi-stable set is in bijection with the similarity classes of rank $1$ $B$-hermitian forms defined on the right $B$-module $B$.

\end{proof}

\begin{remark} 
We have the following observations about the orbit space in the semi-stable set described above.

\begin{enumerate}
    \item  Let $\mathcal{J}=\mathcal{H}(B,\sigma)$ be a Freudenthal division algebra, as described above, and $x,y\in \mathcal{J}\cap B^{\times}$. Then $\sigma_{1}=\text{Int}(x)\circ \sigma$ and $\sigma_{2}=\text{Int}(y)\circ\sigma$ are involutions of the second kind defined on $B$, and $(B,\sigma_{1})$, $ (B,\sigma_{2})$ are isomorphic (as $K$-algebras with involution) if and only if the rank $1$ $B$-hermitian forms $\langle x\rangle_{B}$, $\langle y\rangle_{B}$ are similar (see \cite{HKRT}, Lemma $1$). So, the orbit space in the semi-stable set $\mathcal{J}^{ss}_{k}/(NP(\mathcal{J})\times GL_{1})(k)$ classifies all involutions of these forms defined on $B$.

    \item \textbf{A parametrization of involutions of the second kind:} If $(B,\sigma)$ is a central simple algebra with center $K$ as a quadratic extension of $k$ and $\sigma$ is an involution of the second kind, then for every involution of the second kind $\sigma_{1}$ on $B$ $\exists$ $u\in B^{\times}\cap \mathcal{H}(B,\sigma)$, uniquely determined up to a scalar multiplication such that $\sigma_{1}=\text{Int}(u)\circ \sigma$ (see \cite{KMRT}, Chapter I, Proposition $2.18$). Hence, the orbit space in the semi-stable set described above parametrizes all involutions of the second kind defined on $B$.
   
    \item \textbf{A parametrization of division algebras:} The algebra $\mathcal{H}(B,\sigma)$ is determined up to isomorphism by the algebra $(B,\sigma)$ (see \cite{KMRT}, Chapter IX, Proposition $37.6$). As the above theorem parametrizes all involutions of the second kind defined on $B$, we get a parametrization of all Freudenthal division algebras of dimension $9$ arising from $B$, i.e., the isotopes of $\mathcal{H}(B,\sigma)$.

     \item \textbf{Cohomological interpretation:} The orbit space $\mathcal{J}^{ss}_{k}/(NP(\mathcal{J})\times GL_{1})(k)$ from the above theorem can be identified with the cohomology set $H^{1}(k,GU(B,\sigma))$ (see \cite{KMRT}, Chapter VII).
\end{enumerate}

\end{remark}

\vspace{2 mm}

\noindent \textbf{Division algebras of dimension 27:} Let $\mathcal{J}$ be a Freudenthal division algebra of dimension $27$. Similar to the previous cases, we denote the cubic norm defined on $\mathcal{J}$ by $N_{\mathcal{J}}$ and the group of isometries of this norm by $NP(\mathcal{J})$. Then $(NP(\mathcal{J})\times GL_{1}, \mathcal{J})$ is a PV with $N_{\mathcal{J}}$ as relative invariant, where the $GL_{1}$-action is given by scalar multiplication. The set of semi-stable points $\mathcal{J}^{ss}\subset \mathcal{J}$ contains all non-zero elements in $\mathcal{J}$. We can see from (\cite{KS}, Section $2$, Remark $10$(iv)) that it is enough to consider the action of the structure group (i.e., the group of norm similarities of $\mathcal{J}$) on $\mathcal{J}$ in order to compute the $k$-rational orbit space for this PV. Then using (\cite{GPR}, Theorem $31.22$) we get that the $k$-rational orbit space in the semi-stable set is in bijection with the set of all isomorphism classes of the isotopes of $\mathcal{J}$.

\begin{remark}
    As $\mathcal{J}$ can be any Freudenthal algebra in (\cite{GPR}, Theorem $31.22$), we can get the same interpretation as above for the orbit space $\mathcal{J}^{ss}_{k}/(NP(\mathcal{J})\times GL_{1})(k)$ in the prehomogeneous vector space $(NP(\mathcal{J})\times GL_{1}, \mathcal{J})$ for any Freudenthal algebra $\mathcal{J}$ over $k$.
\end{remark}

A Freudenthal algebra is called \textit{homogeneous} if it is isomorphic to all of its isotopes. Hence we have the following observation for the PV $(NP(\mathcal{J})\times GL_{1}, \mathcal{J})$ associated to a Freudenthal algebra $\mathcal{J}$ over $k$.

\begin{remark}
    A Freudenthal algebra $\mathcal{J}$ over $k$ is homogeneous if and only if the semi-stable set in the PV $(NP(\mathcal{J})\times GL_{1},\mathcal{J})$ is a single orbit over the base field $k$ itself.
\end{remark}

\section{Declarations}
\noindent \textbf{Conflict of interest:} The author states that there is no conflict of interest. 

\vspace{2 mm}

\noindent \textbf{Data availability:} No data sets were generated or analyzed during this work.
\vspace{2 mm}

\noindent \textbf{Funding declaration:} There was no external funding for this research work.

\end{document}